\documentclass{amsart}
\usepackage{graphicx} 

\input{header}

\title{Classification properties for some ternary structures}
\author{Alberto Miguel-Gómez}
\address{Department of Mathematics, Imperial College London, London SW7 2AZ, UK}
\email{a.miguel-gomez22@imperial.ac.uk}

\begin{document}

\maketitle
\begin{abstract}
    We provide a model-theoretic classification of the countable homogeneous $\mathbf{H}_4$-free 3-hypertournament studied by Cherlin, Hubi\v{c}ka, Kone\v{c}n\'y, and Ne\v{s}et\v{r}il. Our main result is that the theory of this structure is $\SOP_3$, $\TP_2$, and $\NSOP_4$. We offer two proofs of this fact: one is a direct proof, and the other employs part of the abstract machinery recently developed by Mutchnik. 
\end{abstract}
\tableofcontents
\section{Introduction}
One of the major current programmes in model theory concerns the extension of methods from stability and simplicity theory to more general classification properties. An example of this program was the development of a structure theory for $\NSOP_1$ theories, mostly carried out by Kaplan and Ramsey in \cite{kaplan2020kimindependence}. The key idea of this development was to generalise the notion of \textit{dividing}, central to the setting of simple theories, to that of \textit{Kim-dividing}, intended to capture the idea of dividing at a generic scale. Kaplan and Ramsey were able to essentially complete the study of $\NSOP_1$ theories using this notion, and the derived notion of \textit{Kim-independence}, given by non-Kim-forking over a model. 

Thus, a natural question is whether we can keep extending these techniques further down the $\NSOP_n$ hierarchy introduced by Shelah in \cite{shelah1996unstable}. Recently, Mutchnik has shown in \cite{mutchnik2023nsop2} that the classes of $\NSOP_2$ and $\NSOP_1$ coincide, and it remains open whether there are any $\NSOP_3$ $\SOP_1$ theories. In contrast, there have been many known examples of natural $\NSOP_4$ $\SOP_3$ theories. The first such examples, already appearing in \cite{shelah1996unstable}, were of a combinatorial nature. 

A contribution to the study of $\NSOP_4$ theories was Conant's study of free amalgamation theories in \cite{conant2017freeamalgamation}, which include many previously known examples such as Henson's generic digraphs (\cite{henson1972digraphs}). In a different direction, Evans and Wong's Hrushovski constructions (\cite{evans2009generic}) provided, in some cases, new examples of strictly $\NSOP_4$ structures. More recently, new algebraic examples of strictly $\NSOP_4$ structures have been found, including d'Elbée, Müller, Ramsey and Siniora's generic $c$-nilpotent Lie algebras over $\F_p$ for $c > 2$ and $p$ prime in \cite{d2025model} and Johnson and Ye's curve-excluding fields in \cite{johnson2025curve}. A common feature of all these examples is the presence of an invariant independence relation defined over models satisfying full existence, symmetry, and stationarity. 

A more recent contribution towards a systematic theory of independence in the context of $\NSOP_4$ theories is due to Mutchnik. In \cite{mutchnik2022conantindependence}, the notion of \textit{Conant-dividing} (also appearing in the literature as \textit{strong Kim-dividing} in \cite{shelah2019kimindependence} and, later, as \textit{universally Kim-dividing} in \cite{kruckman2023new}) is studied in depth. It aims to capture the notion of dividing at a maximally generic scale. A connection is achieved between the related notion of Conant-independence and $\NSOP_4$ by showing that the symmetry of the former implies the latter. Furthermore, in the same paper, Mutchnik generalises the structure theory for $\NSOP_1$ theories relative to a choice of an invariant independence relation satisfying full existence and stationarity over models, and gives criteria for identifying Conant-independence as one of these relative notions of Kim-independence.

In this document, we present the first example of a strictly $\NSOP_4$ structure with no known invariant independent relations satisfying full existence and stationarity, namely, the \textit{$\mathbf{H}_4$-free 3-hypertournament}. More precisely, we prove:
\begin{theorem}\label{thm:main}
    The theory of the countable homogeneous $\mathbf{H}_4$-free 3-hypertournament is $\SOP_3$, $\TP_2$, and $\NSOP_4$.
\end{theorem}
This contrasts with the situation for the other three known homogeneous $3$-hypertournaments, all of which are $\NTP_2$. We also relate our structure to higher-arity versions of stability and $\NIP$, namely, that recently studied by Terry and Wolf in \cite{terry2023higherorder} and by Abd-Aldaim, Conant, and Terry in \cite{abd2023higher} in terms of $\NFOP_n$, and that introduced by Shelah in \cite{shelah2014ndependence} and further developed by Chernikov, Hempel, Palacín, Takeuchi and others in several papers (cf., \cite{chernikov2019n}, \cite{hempel2016n}, \cite{chernikov2019mekler}, \cite{chernikov2021n}) in terms of $\NIP_n$. In particular, we show that the homogeneous $\mathbf{H}_4$-free 3-hypertournament is $\IP_2$ and $\NFOP_3$.

There is a long history of interactions between model theory and the combinatorial study of tournaments, i.e., directed graphs $(V, E)$ such that, for all distinct $a, b \in V$, exactly one of $E(a,b)$ and $E(b,a)$ holds. As a result of this investigation, Lachlan famously classified in \cite{lachlan1984tournaments} the countable homogeneous tournaments into three structures up to isomorphism: the generic tournament, the countable dense linear order without endpoints, and the homogeneous local order.

Around the same time, a combinatorial generalisation of tournaments to higher arities was introduced by Assous in \cite{assous1986hypertournaments}, which was soon slightly modified into the notion of an \textit{$n$-hypertournament}. In an attempt to extend Lachlan's classification of homogeneous tournaments, Cherlin studied the homogeneous \textit{4-constrained} 3-hypertournaments (cf. \cite[\S 23A.4.1]{cherlin2022multitournaments}), showing that there exist only four up to isomorphism. It remains open whether there exist any other countable homogeneous 3-hypertournaments beyond these four cases. Here, we focus on one of these four structures, namely, the $\mathbf{H}_4$-free 3-hypertournament. We introduce its definition in \S \ref{sec:basic-props}, and show some of its basic model-theoretic properties, which include weak elimination of imaginaries and the existence of global types Lascar-invariant over $A$ for $A$ a non-empty set. 

Cherlin's four homogeneous 3-hypertournaments have recently appeared in connection with some fundamental questions in structural Ramsey theory. One of the major open questions in this area is that posed by Bodirsky, Pinsker and Tsankov in \cite{bodirsky2011decidability}: Does every homogeneous structure in a finite relational language have a homogeneous expansion by finitely many relations which is Ramsey? In \cite{cherlin2021hypertournaments}, Cherlin, Hubi\v{c}ka, Kone\v{c}n\'y, and Ne\v{s}et\v{r}il set out to find Ramsey expansions of each of the four homogeneous 4-constrained 3-hypertournaments in finite relational languages, and were able to find them for all except one of the four, namely, the $\mathbf{H}_4$-free 3-hypertournament. Thus, this example serves as motivation for extending the traditional techniques of structural Ramsey theory. 

In \S \ref{sec:first-proof-of-nsop4}, we offer the first proof of Theorem \ref{thm:main}. Nonetheless, the main interest of this example for us lies in the second proof of $\NSOP_4$ that we offer in \S\ref{sec:second-proof-of-nsop4}. As preliminary work towards this proof, in \S\ref{sec:lascar-independence}, we adapt the notion of strong Lascar independence from Tartarotti's master's thesis (\cite{tartarotti2023lascar}) to the context of the Kim-dividing order introduced by Mutchnik in \cite{mutchnik2022conantindependence}. Tartarotti defines strong Lascar independence in terms of minimal extensions of types with respect to the fundamental order, and uses this towards a proof of Lascar's Reconstruction Theorem. Using analogous ideas, we obtain an independence relation we can use to characterise, abstractly, the notions of relative Kim's lemma and strong witnessing property, which play an important role in \cite{mutchnik2022conantindependence}. 

In contrast to most known $\NSOP_4$ examples, including all of those we have mentioned before, there cannot exist any independence relation over models of the theory of the $\mathbf{H}_4$-free $3$-hypertournament satisfying full existence, symmetry, and stationarity. In our second proof of $\NSOP_4$ (\S\ref{sec:second-proof-of-nsop4}), we show that:
\begin{theorem}
    In the $\mathbf{H}_4$-free $3$-hypertournament, there is a non-stationary independence relation $\ind^{\textnormal{hti}}$ satisfying full existence and the relative Kim's lemma. In particular, Conant-independence coincides with $\ind^{\textnormal{a}}$.
\end{theorem}
Thus, although our proof uses some of the concepts and results that Mutchnik introduces in \cite{mutchnik2022conantindependence}, the specific methods applied to examples in that paper use the existence of an independence relation satisfying monotonicity, full existence, and stationarity over models, which we have not been able to find in the $\mathbf{H}_4$-free 3-hypertournament. We conjecture that, in fact, there is no such independence relation defined over models of this theory. However, regardless of the outcome of this conjecture, the present example remains a theoretical novelty in the context of $\NSOP_4$ as the first application of Mutchnik's concepts with a non-stationary independence relation.
\subsection*{Acknowledgments}
The present work was completed during my PhD at Imperial College London supported by an EPSRC scholarship. I am deeply grateful to my supervisors, David Evans and Charlotte Kestner, for their guidance, suggestions, and constant support. I would also like to thank Nicholas Ramsey, Scott Mutchnik, and Paolo Marimon for the enlightening discussions on the results of this paper, and Christian d'Elbée for the encouragement to write this article. Finally, I thank the anonymous referee for a very careful reading of this paper and many useful suggestions to improve its quality.
\subsection*{Keywords} $\NSOP_4$, homogeneous hypertournaments, Conant-independence
\section{Conventions and preliminaries}
Since we need to keep track of elements and tuples in many of the proofs concerning 3-hypertournaments, we use $a$ to denote an element, the bar notation $\bar{a}$ to denote a tuple, and $A$ to denote a set. For an $n$-tuple $\bar{a}$, we write $\bar{a} = (a_0, \dots, a_{n-1})$. As usual, we denote by $AB$ the union of sets $A \cup B$.

For an ordinal $\alpha$, we denote by $\alpha^{< \omega}$ the tree of finite sequences of elements of $\alpha$, and we denote its partial order by $\trianglelefteq$. For $\eta \in \alpha^{\omega}$ and $i \in \omega$, we write $\eta|_i$ for the restriction of $\eta$ to the first $i$ entries. For $\eta, \nu \in \alpha^{< \omega}$, we write $\eta^\frown \nu$ to denote the concatenation of $\eta$ and $\nu$ as sequences.
\subsection{Fraïssé's Theorem}
We quickly review some of the main concepts and results from the theory of homogeneous structures that we employ throughout the present work. More details can be found in, e.g., \cite[\S\S 2.6-2.8]{cameron1990oligomorphic}.
\begin{definition}
    Let $\mathcal{L}$ be a relational language. 
    \begin{enumerate}[(i)]
        \item We say an $\mathcal{L}$-structure $M$ is \textbf{homogeneous} if, for all finite substructures $A, B \subset M$, any isomorphism $f \colon A \to B$ extends to an automorphism $g \colon M \to M$. 
        \item The \textbf{age} of $M$ is the class of all $\mathcal{L}$-structures isomorphic to finite substructures of $M$. 
        \item We say a class $\mathcal{C}$ of finite $\mathcal{L}$-structures has the \textbf{amalgamation property}, or \textbf{AP}, if, for all $A, B_1, B_2 \in \mathcal{C}$ and embeddings $f_i \colon A \to B_i$ for $i = 1,2$, there exist some $C \in \mathcal{C}$ and embeddings $g_i \colon B_i \to C$ for $i = 1,2$ making the following diagram commute:
        
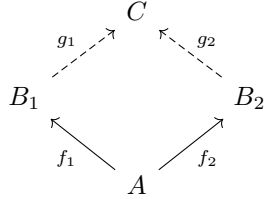
\begin{equation*}
            \begin{tikzcd}
                & C & \\
                B_1 \arrow[ru, dashed, "g_1"] & & B_2 \arrow[lu, swap, dashed, "g_2"] \\
                & A \arrow[lu, "f_1"] \arrow[ru, swap, "f_2"] & 
            \end{tikzcd}
        \end{equation*}
        We say that $\mathcal{C}$ has the \textbf{strong amalgamation property}, or \textbf{SAP}, if the above holds and, in addition, whenever $b_i \in B_i$ for $i=1,2$ are such that $g_1(b_1) = g_2(b_2)$, there is some $a \in A$ such that $b_1 = f_1(a)$ and $b_2 = f_2(a)$. 
    \end{enumerate}
\end{definition}
To prove that a class has AP, it is enough to show this for $\size{B_i \setminus f_i(A)} = 1$ for $i = 1,2$.

Let us note in passing that many model-theorists call a structure as in (i) above \textbf{ultrahomogeneous} instead. However, in what follows, we adopt the convention from \cite{cherlin2021hypertournaments} and use ``homogeneous'' as above.
\begin{definition}
    Let $\mathcal{L}$ be a relational language. We say a class $\mathcal{C}$ of finite $\mathcal{L}$-structures is a (resp., \textbf{strong}) \textbf{amalgamation class} if it is closed under substructures and isomorphisms, has countably many isomorphism classes, and has AP (resp., SAP). 
\end{definition}
\begin{fact}[Fraïssé's Theorem] \label{fact:fraisse}
    Let $\mathcal{L}$ be a relational language and $\mathcal{C}$ be a class of finite $\mathcal{L}$-structures. Then $\mathcal{C}$ is an amalgamation class iff there is a countable homogeneous $\mathcal{L}$-structure $M$ such that $\mathcal{C}$ is the age of $M$. This $M$ is unique up to isomorphism. (We call $M$ the \textbf{Fraïssé limit} of $\mathcal{C}$.)
\end{fact}
\begin{fact}[\protect{cf. \cite[2.22]{cameron1990oligomorphic}}] \label{fact:fraisse-limit-omega-cat}
    Let $M$ be a countable homogeneous structure and let $T = \textnormal{Th}(M)$. Then $T$ is $\omega$-categorical and has quantifier elimination.
\end{fact}
\begin{fact}[\protect{cf. \cite[2.15]{cameron1990oligomorphic}}] \label{fact:trivial-acl-for-sap}
    Let $M$ be a countable homogeneous structure in a relational language. Then the age of $M$ is a strong amalgamation class iff $\acl(A) = A$ for all finite $A \subset M$. 
\end{fact}
\subsection{Generalised stability theory}
From now on, let $T$ be a complete theory and $\M \models T$ a monster model, i.e., a sufficiently saturated and strongly homogeneous model. As usual, we assume that all elements, tuples, and sets are small and embed into $\M$. Types defined over $\M$ are called \textbf{global}. Let us recall the relevant definitions of the classification properties we use:
\begin{definition}
Let $T$ be a complete theory.
    \begin{enumerate}[(i)]
        \item We say $T$ has the \textbf{tree property} (or is \textbf{TP}) if there are a formula $\phi(\bar{x}, \bar{y})$, $k \in \omega$, and a tree $(\bar{a}_\eta)_{\eta \in \omega^{<\omega}}$ such that:
        \begin{itemize}
            \item for all $\eta \in \omega^{\omega}$, the set $\{\phi(\bar{x}, \bar{a}_{\eta|_i}) : i \in \omega\}$ is consistent modulo $T$;
            \item for all $\eta \in \omega^{<\omega}$, any $k$-subset of $\{\phi(\bar{x}, \bar{a}_{\eta^\frown \langle i\rangle}) : i \in \omega\}$ is inconsistent modulo $T$.
        \end{itemize}
        Otherwise, we say $T$ is \textbf{simple}.
        \item We say $T$ has the \textbf{tree property of the second kind} (or is $\mathbf{TP}_2$) if there are a formula $\phi(\bar{x},\bar{y})$, $k \in \omega$, and an array $(\bar{a}_{i,j})_{i, j \in \omega}$ such that:
        \begin{itemize}
            \item for all $f \colon \omega \to \omega$, the set $\{\phi(\bar{x}, \bar{a}_{i,f(i)}) : i \in \omega\}$ is consistent modulo $T$;
            \item for all $i \in \omega$, any $k$-subset of $\{\phi(\bar{x}, \bar{a}_{i,j}) : j \in \omega\}$ is inconsistent modulo $T$. 
        \end{itemize}
        Otherwise, we say $T$ is $\mathbf{NTP}_2$.
        \item We say $T$ has the \textbf{1-strong order property} (or is $\mathbf{SOP}_1$) if there are a formula $\phi(\bar{x}, \bar{y})$ and a tree $(\bar{a}_\eta)_{\eta \in 2^{< \omega}}$ such that:
        \begin{itemize}
            \item for all $\eta \in 2^{\omega}$, the set $\{\phi(\bar{x}, \bar{a}_{\eta|_i}) : i \in \omega\}$ is consistent modulo $T$;
            \item for all $\eta, \nu \in 2^{< \omega}$ with $\nu^\frown \langle 0 \rangle \trianglelefteq \eta$, the set $\{\phi(\bar{x}, \bar{a}_{\eta}), \phi(\bar{x}, \bar{a}_{\nu^\frown \langle 1 \rangle})\}$ is inconsistent modulo $T$.
        \end{itemize}
        Otherwise, we say $T$ is $\mathbf{NSOP}_1$.
        \item Let $3 \leq n \in \omega$. We say $T$ has the \textbf{$n$-strong order property} (or is $\mathbf{SOP}_n$) if there are a formula $\phi(\bar{x}, \bar{y})$ with $\size{\bar{x}} = \size{\bar{y}}$ and a sequence $(\bar{a}_i)_{i \in \omega}$ such that:
        \begin{itemize}
            \item $\models \phi(\bar{a}_i, \bar{a}_j)$ for all $i < j$;
            \item the set $\{\phi(\bar{x}_0, \bar{x}_1), \phi(\bar{x}_1, \bar{x}_2), \dots, \phi(\bar{x}_{n-2}, \bar{x}_{n-1}), \phi(\bar{x}_{n-1}, \bar{x}_0)\}$ is inconsistent modulo $T$.
        \end{itemize}
        Otherwise, we say $T$ is $\mathbf{NSOP}_n$.
    \end{enumerate}
\end{definition}
Another recent programme in model theory consists in extending the traditional binary classification-theoretic properties to higher arities. Two of the most fruitful notions in this direction are the following, introduced by Shelah in \cite{shelah2014ndependence} and by Terry and Wolf in \cite{terry2023higherorder} (extended by Abd-Aldaim, Conant, and Terry in \cite{abd2023higher}), respectively:
\begin{definition} Let $T$ be a complete theory.
    \begin{enumerate}[(i)]
        \item Let $0 < n \in \omega$. We say $T$ has the \textbf{$n$-independence property} (or is $\mathbf{IP}_n$) if there are a formula $\phi(\bar{x}_0, \dots, \bar{x}_{n-1}, \bar{y})$ and sequences $(\bar{a}_{0,i})_{i \in \omega}, \dots, \linebreak (\bar{a}_{n-1,i})_{i<\omega}$ and $(\bar{b}_I)_{I \subseteq \omega^n}$ such that
        \begin{equation*}
            \models \phi(\bar{a}_{0,i_0}, \dots, \bar{a}_{n-1, i_{n-1}}, \bar{b}_I) \iff (i_0, \dots, i_{n-1}) \in I
        \end{equation*}
        for all $i_0, \dots, i_{n-1} \in \omega$ and $I \subseteq \omega^n$. Otherwise, we say $T$ is $\mathbf{NIP}_n$. When $n = 1$, we just say $T$ is $\mathbf{NIP}$.
        \item Let $2 \leq n \in \omega$. We say $T$ has the \textbf{$n$-functional order property} (or is $\mathbf{FOP}_n$) if there are a formula $\phi(\bar{x}_0, \dots, \bar{x}_n)$ and sequences $(\bar{a}_f)_{f \colon \omega^{n-1} \to \omega}$, $(\bar{b}_{0,i})_{i \in \omega}, \dots, (\bar{b}_{n-1, i})_{i \in \omega}$ such that
        \begin{equation*}
            \models \phi(\bar{a}_f, \bar{b}_{1,i_1}, \dots, \bar{b}_{n, i_{n}}) \iff i_{n} \leq f(i_1, \dots, i_{n-1})
        \end{equation*}
        for all $i_1, \dots, i_n \in \omega$ and $f \colon \omega^{n-1} \to \omega$. Otherwise, we say $T$ is $\mathbf{NFOP}_n$.
    \end{enumerate}
\end{definition}
The following is a diagram of the (known) implications between the above notions. This combines many results from the literature; a good recopilation of most of them containing all the relevant references is \cite{conant2012dividing} (the relation between $\NIP_n$ and $\NFOP_n$ is more recent and appears in \cite{abd2023higher}).
\begin{equation*}
        \begin{tikzcd}[sep=small]
            \NFOP_2 \arrow[r, Rightarrow] & \NIP_2 \arrow[r, Rightarrow] & \NFOP_3 \arrow[r, Rightarrow] & \NIP_3 \arrow[r, Rightarrow] & \cdots \\
            \textnormal{\NIP} \arrow[u, Rightarrow] \arrow[r, Rightarrow] & \NTP_2 \\
            \textnormal{stable} \arrow[u, Rightarrow] \arrow[r, Rightarrow] & \textnormal{simple} \arrow[u, Rightarrow] \arrow[r, Rightarrow] & \NSOP_1 \arrow[r, Rightarrow] & \NSOP_3 \arrow[r, Rightarrow] & \NSOP_4 \arrow[r, Rightarrow] & \cdots
        \end{tikzcd}
    \end{equation*}
It is open whether there exists an $\NSOP_n$ theory for any $n \geq 3$ which is $\NTP_2$ but not simple. 
\subsection{A primer on independence relations}
Formal definitions of the notion of an ``independence relation'' are commonplace in the literature (see, e.g., \cite{adler2005explanation}). However, these definitions fall short of capturing more recent developments within model theory, such as the notion of Kim-independence in $\NSOP_1$ theories or that of Conant-independence that we study below. For this reason, we follow \cite{delbee2023axiomatic} and freely use the term \textbf{independence relation} to denote an $\Aut(\M)$-invariant ternary relation $\ind$ on subsets of our fixed monster model, i.e., $A \ind_C B$ iff $\sigma(A) \ind_{\sigma(C)} \sigma(B)$ for all $\sigma \in \Aut(\M)$. (This use of the term is already prefigured in \cite{chernikov2012forking}, although the authors of that paper refer to this as a ``pre-independence relation.'')

There are several properties that an independence relation may satisfy, which already appear, in some form, in \cite{shelah1978classification}. Their explicit versions in this abstract setting can be traced back to \cite{baldwin1988fundamentals}:
\begin{itemize}
    \item \textit{Left monotonicity}: For all $A, B, C \subset \M$, if $A \ind_C B$ and $A' \subseteq A$, then $A' \ind_C B$.
    \item \textit{Right monotonicity}: For all $A, B, C \subseteq \M$, if $A \ind_C B$ and $B' \subseteq B$, then $A \ind_C B'$.
    \item \textit{Right transitivity}: For all $A \subset \M$ and $D \subseteq C \subseteq B\subset \M$, if $A \ind_C B$ and $A \ind_D C$, then $A \ind_D B$.
    \item \textit{Existence}: For all tuples $\bar{a} \subset \M$ and sets $C \subset \M$, $\bar{a} \ind_C C$.
    \item \textit{Full existence}: For all tuples $\bar{a} \subset \M$ and sets $B, C \subseteq \M$, there is $\bar{a}' \equiv_C \bar{a}$ such that $\bar{a}' \ind_C B$.
    \item \textit{Left extension}: For all $A, A', B, C \subset \M$, if $A \ind_C B$ and $A \subseteq A'$, there is some $B' \equiv_{AC} B$ such that $A' \ind_C B'$.
    \item \textit{Right extension}: For all $A, B, B', C \subseteq \M$, if $A \ind_C B$ and $B \subseteq B'$, there is some $A' \equiv_{BC} A$ such that $A' \ind_C B'$.
    \item \textit{Stationarity over models}: For $A, A', B \subset \M$ and $M \prec \M$, if $A \ind_M B$, $A' \ind_M B$, and $A \equiv_M A'$, then $A \equiv_{MB} A'$. 
\end{itemize}
We often say that $\ind$ satisfies \textbf{monotonicity} if it satisfies both left and right monotonicity.
We call $\ind$ an independence relation \textbf{over models} if its base can only be a model, i.e., if $A \ind_C B$ is only defined if $C$ is a model of $T$.
\begin{example} \label{ex:inv-ind-rel-properties}
    As is standard in the literature, let us define, for tuples $\bar{a}$, $\bar{b}$ and sets $C$, $\bar{a} \ind^{\text{i}}_C \bar{b}$ iff $\tp(\bar{a}/C\bar{b})$ has a global extension Lascar-invariant over $C$. Restricting $\ind^{\text{i}}$ to be defined only over models, the definition is equivalent to the existence of a global $M$-invariant extension of $\tp(\bar{a}/M\bar{b})$. In any theory $T$, $\ind^{\text{i}}$ satisfies invariance and monotonicity, among other properties not discussed here.
\end{example}
\begin{notation}
    Given two independence relations $\ind^1$ and $\ind^2$, we write $\ind^1 \implies \ind^2$ if, whenever $A \ind_C^1 B$, we have $A \ind_C^2 B$.
\end{notation}
We also recall some operations on abstract independence relations which will be useful for the later sections. The first one comes, in its explicit form, from Adler's thesis (\cite[Definition 1.16]{adler2005explanation}):
\begin{definition}
    Given an independence relation $\ind$, we define $\ind^*$ by
    \begin{equation*}
        A \ind_C^* B \iff \text{ for all } B' \supseteq B, \text{ there exists } A' \equiv_{BC} A \text{ such that } A' \ind_C B'.
    \end{equation*}
\end{definition}
It is immediate from the definition that:
\begin{itemize}
    \item $\ind^* \implies \ind$.
    \item If $\ind^0 \implies \ind$ and $\ind^0$ satisfies right extension, then $\ind^0 \implies \ind^*$.
\end{itemize}
\begin{fact}[\protect{\cite[Lemma 1.17]{adler2005explanation}}] \label{fact:ind-star-props} If $\ind$ satisfies right monotonicity, then $\ind^*$ satisfies right monotonicity and right extension. If, in addition, $\ind$ satisfies left monotonicity, then so does $\ind^*$. 
\end{fact}
\begin{example}
    The classical example of this operation is that the non-forking independence relation, denoted $\ind^{\textnormal{f}}$, is the result of applying the above operation to the non-dividing independence relation, denoted $\ind^{\textnormal{d}}$. Explicitly: $\ind^{\textnormal{f}} = (\ind^{\textnormal{d}})^*$.
\end{example}
The second operation has been explicitly defined by d'Elbée (\cite[Definition 3.2.9]{delbee2023axiomatic}):
\begin{definition}
    Given an independence relation $\ind$, we define $\ind^{\textnormal{opp}}$ by
    \begin{equation*}
        A \ind_C^{\text{opp}} B \iff B \ind_C A.
    \end{equation*}
\end{definition}
\begin{remark} \label{rem:ind-opp-props}
    It is clear from the definition that, if $\ind$ satisfies left monotonicity/left extension (resp., right monotonicity/right extension), then $\ind^{\text{opp}}$ satisfies right monotonicity/right extension (resp., left monotonicity/left extension).
\end{remark}
\begin{example}
    The coheir independence relation, denoted $\ind^{\textnormal{u}}$, is the result of applying the opp operation to the heir independence relation, denoted $\ind^{\textnormal{h}}$. Explicity: $\ind^{\text{u}} = (\ind^{\text{h}})^{\text{opp}}$.
\end{example}
\subsection{Relative Kim independence}
Given an independence relation $\ind$, we can relativise the notions appearing in the theory of $\NSOP_1$, in the way which can be found in \cite{mutchnik2022conantindependence} and originates in \cite{adler2005explanation}. Let us start with the following notion from the theory of abstract independence relations (cf. \cite[Definition 2.26]{chernikov2012forking}):
\begin{definition}
    Let $T$ be a complete theory, and $\ind$ be an independence relation over subsets of $\M \models T$. We say a global type $q$ is \textbf{$\ind$-free} over $M$ if, for all $B \supset M$ and $\bar{a} \models q|_B$, we have $\bar{a} \ind_M B$.
\end{definition}
\begin{example}
    A global type is $\ind^{\text{i}}$-free over $M$ iff it is $M$-invariant.
\end{example}
\begin{definition}
Let $T$ be a complete theory, $M \models T$, and $\bar{a}$, $\bar{b}$ be tuples.
    \begin{enumerate}[(i)] 
        \item We say that $(\bar{b}_i)_{i \in \omega}$ is an \textbf{$\ind$-Morley sequence} over $M$ if, for all $i \in \omega$, $\bar{b}_i \equiv_M \bar{b}_0$ and $\bar{b}_i \ind_M \bar{b}_{<i}$. For any global type $q$ which is $\ind$-free over $M$, we say $(\bar{b}_i)_{i \in \omega}$ is a \textbf{Morley sequence in $q$ over $M$} if $\bar{b}_i \models q|_{M\bar{b}_{<i}}$ for all $i < \omega$. Every Morley sequence in a global type $\ind$-free over $M$ is $\ind$-Morley over $M$.
        \item We say a formula $\phi(\bar{x}, \bar{b})$ \textbf{$\ind$-Kim-divides} over $M$ if there is some global extension $q \supset \tp(\bar{b}/M)$ $\ind$-free over $M$ and a Morley sequence $(\bar{b}_i)_{i \in \omega}$ in $q$ over $M$ with $\bar{b}_0 = \bar{b}$ such that the set $\{\phi(\bar{x}, \bar{b}_i) : i \in \omega\}$ is inconsistent. 
        \item We say a formula $\phi(x, b)$ \textbf{$\ind$-Kim-forks} over $M$ if there exist formulas $\psi_i(x, c_i)$ for $i < n$ such that $\phi(x, b) \vdash \bigvee_{i < n} \psi_i(x, c_i)$ and each $\psi_i(x, c_i)$ $\ind$-Kim-divides over $M$.
        \item We say $\bar{a}$ is \textbf{$\ind$-Kim-independent} from $\bar{b}$ over $M$ if $\tp(\bar{a}/M\bar{b})$ does not contain any formula that $\ind$-Kim-forks over $M$.
    \end{enumerate}
\end{definition}
\begin{example} 
    We call Morley sequences in global $M$-invariant types \textbf{$M$-invariant Morley sequences}. In this paper, we refer to $\ind^{\textnormal{i}}$-Kim-independence simply as \textbf{Kim-independence}.
\end{example}
Recently, Mutchnik in \cite{mutchnik2022conantindependence} has generalised the fundamental order introduced by Poizat \cite{lascar1979forking} to Kim-dividing, extending the work of Ben Yaacov and Chernikov in \cite{benyaacov2014independence}, who already generalised it to dividing. Our presentation follows that of the fundamental order given by Pillay in \cite{pillay1983stability}.
\begin{definition}
    Let $T$ be a complete theory, $M \models T$, and $r \in S(M)$.
    \begin{enumerate}[(i)]
        \item Let $p$ be a global $M$-invariant extension of $r$. The \textbf{Kim-dividing class of $p$}, denoted $\cl_K(p)$, is defined to be the set of formulas $\phi(\bar{x},\bar{y}) \in \mathcal{L}_M$ such that $\phi(\bar{x}, \bar{a})$ is consistent for some (equiv., any) realisation $\bar{a} \models p|_M$ and there is $(\bar{a}_i)_{i \in \omega}$ with $\bar{a}_i \models p|_{M\bar{a}_{<i}}$ for all $i\in \omega$ and $\{\phi(\bar{x}, \bar{a}_i) : i \in \omega\}$ inconsistent. We may equivalently ask that $\{\phi(\bar{x}, \bar{a}_i) : i \in \omega\}$ is inconsistent for all $(\bar{a}_i)_{i \in \omega}$ as above.
        \item Given global $M$-invariant extensions $p, q$ of $r$, we write $p \leq_K q$ if $\cl_K(p) \subseteq \cl_K(q)$. We say $p$ is \textbf{least in the Kim-dividing order}, or \textbf{$\leq_K$-least}, if $p \leq_K q$ for all global $M$-invariant extensions $q$ of $r$.
    \end{enumerate}
\end{definition}
Mutchnik relativises Kim's lemma to a choice of independence relation:
\begin{definition}
    We say an independence relation $\ind$ defined over models satisfies the \textbf{relative Kim's lemma} if, for all global types $q$ and models $M \models T$, if $q$ is $\ind$-free over $M$, then $q$ is a $\leq_K$-least extension of $q|_M$. 
\end{definition}
\begin{remark} \label{rem:rel-kim-lemma-and-invariance}
    Note that, by the definition of the Kim-dividing order, if $\ind$ satisfies the relative Kim's lemma, then every global type that is $\ind$-free over $M$ must be $M$-invariant.
\end{remark}
\begin{example} \label{ex:kim-lemma-for-nsop1}
    Another way of restating Kaplan and Ramsey's version of Kim's lemma for $\NSOP_1$ theories (cf. \cite[Theorem 3.16]{kaplan2020kimindependence}) is the following (and hence the terminology): a theory $T$ is $\NSOP_1$ iff $\ind^{\text{i}}$ satisfies the relative Kim's lemma.
\end{example}
\begin{remark}
    In the special case where $\ind$ satisfies stationarity over models, some of the previous definitions may be simplified. This is due to the following easy observations:
    \begin{itemize}
        \item If $\ind$ satisfies full existence and stationarity over models, then every type $p \in S(M)$ has a unique global extension that is $\ind$-free over $M$. Such an extension is additionally $M$-invariant.
        \item If $\ind$ satisfies stationarity over models, then every $\ind$-Morley sequence over $M$ is $M$-indiscernible.
    \end{itemize}
    In the context of Mutchnik's work on relative Kim-independence (cf., \cite{mutchnik2022conantindependence}), the relevant independence relation satisfies stationarity over models by assumption, and hence the above remarks apply. As we will see in \S6, we want to use this machinery in the context of a non-stationary independence relation. The observations in this remark will not be used later on; they serve as a point of comparison between these two approaches.
\end{remark}
\section{Basic properties} \label{sec:basic-props}
\subsection{Definition of the \texorpdfstring{$\mathbf{H}_4$}{H4}-free 3-hypertournament} The definitions and results included in this subsection can be found, explicitly or implicitly, in \cite{cherlinunpublished} and \cite{cherlin2021hypertournaments}. Let $\mathcal{L} = \{R\}$, where $R$ is a ternary relation symbol.
\begin{definition}
    Let $T$ be the $\mathcal{L}$-theory axiomatised by the following formulas:
    \begin{enumerate}[(i)]
        \item $\forall x_1 x_2 x_3(R(x_1,x_2,x_3) \to \bigwedge_{i \neq j} x_i \neq x_j)$.
        \item $\forall x_1 x_2 x_3 (R(x_1, x_2, x_3) \leftrightarrow R(x_2, x_3, x_1) \leftrightarrow R(x_3, x_1, x_2))$.
        \item $\forall x_1 x_2 x_3 (\bigwedge_{i \neq j} x_i \neq x_j \to (\neg R(x_1, x_2, x_3) \leftrightarrow R(x_3, x_2, x_1)))$.
    \end{enumerate}
    We call any model $M \models T$ a \textbf{$3$-hypertournament}.
\end{definition}
The general classification of all countable homogeneous $3$-hypertournaments is still open. As a partial answer to this problem, Cherlin's result classifies those homogeneous $3$-hypertournaments that arise from looking at the possible structures that can occur on any subset of four points contained in them.
\begin{remark}
    Let $M$ be a $3$-hypertournament. Given a linear order $\leq$ on $M$, define a $3$-uniform hypergraph $\widehat{M}$ on $M$ in $\mathcal{L}$ by declaring, for $a \leq b \leq c$, $\{a,b,c\} \in R^{\widehat{M}}$ iff $(a,b,c) \in R^M$. Formally, a linear order on $M$ determines an isomorphism of categories between the age of $M$ with embeddings to the category of $3$-uniform hypergraphs on finite subsets of $M$ with embeddings.
\end{remark}
There are three $4$-point substructures of a $3$-hypertournament up to isomorphism:
\begin{itemize}
    \item $\mathbf{C}_4$: there is an order $\leq$ on $C_4$ (the underlying set of four points) such that $\widehat{\mathbf{C}}_4$ is a complete $3$-uniform hypergraph. In different orders, $\widehat{\mathbf{C}}_4$ might have no hyperedges at all, or exactly two hyperedges which intersect at two vertices that are adjacent relative to the ordering on the four points. 
    \item $\mathbf{O}_4$: for any linear order $\leq$ on $O_4$, $\widehat{\mathbf{O}}_4$ has an odd number of hyperedges. 
    \item $\mathbf{H}_4$: for any linear order $\leq$ on $H_4$, $\widehat{\mathbf{H}}_4$ has exactly two hyperedges intersecting in vertices $a < b$ such that there is a unique $c \in H_4$ with $a < c < b$.
\end{itemize}
For any $\mathbf{A} \in \{\mathbf{C}_4, \mathbf{O}_4, \mathbf{H}_4\}$, given a $3$-hypertournament $M$ and a substructure $B \subseteq M$, we say $B$ is an \textbf{$\mathbf{A}$-structure} if it is isomorphic to $\mathbf{A}$, and we call $M$ $\mathbf{A}$\textbf{-free} if it contains no $\mathbf{A}$-structures.
\begin{definition}
    Let $\mathcal{C}$ be a class of finite $3$-hypertournaments. We say $\mathcal{C}$ is \textbf{$4$-constrained} if there is $S \subseteq \{\mathbf{C}_4, \mathbf{O}_4, \mathbf{H}_4\}$ such that a finite $3$-hypertournament $M \in \mathcal{C}$ iff any substructure $B \subseteq M$ with $\size{B} = 4$ is an $\mathbf{A}$-structure for some $\mathbf{A} \in S$.
\end{definition}
Cherlin proves there are exactly four $4$-constrained classes of $3$-hypertournaments that are (strong) amalgamation classes in \cite{cherlinunpublished}. Here, we focus on just one of these classes, namely, that of finite $3$-hypertournaments omitting $\mathbf{H}_4$.
\begin{remark}\label{rem:h4-free-criterion}
    Since we focus on $\mathbf{H}_4$ from now on, let us note that we may also define it as, up to isomorphism, a structure on four points $\{a, b, c, d\}$ such that
    \begin{equation*}
        \models R(a, b, c) \wedge R(a, c, d) \wedge R(a, d, b) \wedge R(b, d, c).
    \end{equation*}
    We often use without mention the following easy observation: if $a, b, c, d \in M$ are elements of a 3-hypertournament such that $M \models R(a,b,c) \wedge R(a,b,d)$, then $\{a,b,c,d\}$ is $\mathbf{H}_4$-free.
\end{remark}
We include the proof of strong amalgamation for the sake of completeness:
\begin{lemma}\label{lem:h4-free-is-strong-amalgamation-class}
    The class of finite $\mathbf{H}_4$-free $3$-hypertournaments is a strong amalgamation class.
\end{lemma}
\begin{proof}
    We may assume that our amalgamation problem is of the form $Ab_1 \leftarrow A \rightarrow Ab_2$ for $b_1, b_2 \notin A$. Let $\leq$ be a linear order on $Ab_1b_2$ such that $b_1 < A < b_2$. This generates an amalgamation problem in the category of finite $3$-uniform hypergraphs, and so we can freely amalgamate to obtain a solution $\widehat{C}$. Applying the inverse functor, we obtain a finite $3$-hypertournament $C$ which is in our class. Indeed, by construction, the only potential $\mathbf{H}_4$-structure must be of the form $\{a_1, a_2, b_1, b_2\}$ with $a_1, a_2 \in A$. Since $\widehat{C}$ is obtained by free amalgamation, no hyperedges involving both $b_1$ and $b_2$ appear in $\widehat{C}$, and so it follows that $C \models R(b_1, b_2, a_1) \wedge R(b_1, b_2, a_2)$, which implies that $\{a_1, a_2, b_1, b_2\}$ is $\mathbf{H}_4$-free. 
\end{proof}
By Fact \ref{fact:fraisse}, the Fraïssé limit of the class in Lemma \ref{lem:h4-free-is-strong-amalgamation-class} exists, and we denote its theory by $T_{\mathbf{H}_4\textnormal{-free}}$. By Facts \ref{fact:fraisse-limit-omega-cat} and \ref{fact:trivial-acl-for-sap}, $T_{\mathbf{H}_4\textnormal{-free}}$ is $\omega$-categorical and has quantifier elimination and trivial algebraic closure. Often, we refer to the countable model of $T_{\mathbf{H}_4\textnormal{-free}}$ as \textit{the} $\mathbf{H}_4$-free $3$-hypertournament.

The three remaining countable homogeneous $3$-hypertournaments are determined by the constrained classes they come from. These are:
\begin{itemize}
    \item The \textbf{generic 3-hypertournament}, which is the Fraïssé limit of the $4$-constrained class determined by $S = \{\mathbf{C}_4, \mathbf{O}_4, \mathbf{H}_4\}$. Using, e.g., disjoint 3-amalgamation, one can show that the generic 3-hypertournament is supersimple with trivial forking (see \cite[Theorem 3.14]{kruckman2019disjoint}). 
    \item The \textbf{even 3-hypertournament}, which is the Fraïssé limit of the $4$-constrained class determined by $S = \{\mathbf{C}_4, \mathbf{H}_4\}$. As before, one can show that the even 3-hypertournament is supersimple with trivial forking. 
    \item The \textbf{cyclic 3-hypertournament}, which is the Fraïssé limit of the $4$-constrained class determined by $S = \{\mathbf{C}_4\}$. One can show that the cyclic 3-hypertournament is bi-interpretable with $(\Q, \text{cyc})$, and hence it is distal and dp-minimal. 
\end{itemize}
The goal of this paper is to provide similar classification results for the $\mathbf{H}_4$-free $3$-hypertournament.
\subsection{Existence of invariant extensions}
In this section, we prove that $T_{\mathbf{H}_4\textnormal{-free}}$ is locally $\ind^\text{i}$-extensible but not $\ind^{\text{i}}$-extensible, properties which we define below. During the course of the proof, some other basic model-theoretic properties of $T_{\mathbf{H}_4\textnormal{-free}}$ are established, such as weak elimination of imaginaries. From now on, let $\M \models T_{\mathbf{H}_4\textnormal{-free}}$ be a monster model.
\begin{definition}[\protect{cf. \cite{chernikov2012forking}}]
    Let $T$ be a complete theory. 
    \begin{enumerate}[(i)]
        \item We say $T$ is \textbf{$\ind^{\textnormal{i}}$-extensible} if $\ind^{\textnormal{i}}$ satisfies existence, i.e., every type $p \in S(A)$ has a global extension $q$ which is Lascar-invariant over $A$.
        \item We say $T$ is \textbf{locally $\ind^{\textnormal{i}}$-extensible} if it is $\ind^{\textnormal{i}}$-extensible after adding finitely many parameters. 
    \end{enumerate}
\end{definition}
\begin{remark}
    For (local) $\ind^\textnormal{i}$-extensibility, since $\ind^{\text{i}}$ satisfies right transitivity and monotonicity (cf. Example \ref{ex:inv-ind-rel-properties}), it suffices to check the condition from the definition for 1-types.  
\end{remark}
\begin{lemma} \label{lem:wei}
    $T_{\mathbf{H}_4\textnormal{-free}}$ has weak elimination of imaginaries, i.e., for every imaginary $e$, there is some real $\bar{c}$ such that $e \in \dcl^\eq(\bar{c})$ and $\bar{c} \in \acl(e)$. 
\end{lemma}
\begin{proof}
    Let $M \models T_{\mathbf{H}_4\textnormal{-free}}$ be countable. By \cite[Lemma 16.17]{poizat2000modeltheory}, it suffices to show that
    \begin{equation} \label{eq:aut-of-intersection}
        H := \langle \Aut(M/A), \Aut(M/B)\rangle = \Aut(M/A \cap B)
    \end{equation}
    for finite algebraically closed $A, B \subset M$. Since $M$ has trivial algebraic closure, by \cite[Proposition 4.3]{li2018simplicity}, it is enough to prove (\ref{eq:aut-of-intersection}) for $A = Ca$, $B = Cb$ for some finite set $C \subseteq M$ and distinct elements $a, b \in M$. We may assume $a,b \notin C$.
    \begin{nclaim}
        If $b' \equiv_C b$, then there is $a' \models \tp(a/Cb)$ such that $b' \equiv_{Ca'} b$.
    \end{nclaim}
    \begin{proof}[Proof of Claim 1]
        Write $C = \{c_1, \dots, c_n\}$ and $p(x, b, \bar{c}) := \tp(a/b\bar{c})$, and let $b' \models \tp(b/\bar{c})$. First note that, if $b' = b$, then we can pick $a' = a$. Hence, assuming that $b' \neq b$, we need to show that $p(x, b, \bar{c}) \cup \{R(x, b, c_i) \leftrightarrow R(x, b', c_i) : i \in [n]\}$ is consistent. We split the proof of this claim into two cases. In both, our strategy is to build a finite structure that satisfies the formulas in the set. 

        \underline{Case 1}: $b' = a$.

        Let $\Sigma(x,b,a)$ be the extension of $p(x,b,\bar{c})$ given by adding the following formulas:
        \begin{enumerate}[(i)]
            \item $R(x,a,c_i) \leftrightarrow R(a,b,c_i)$ for all $i \in [n]$. 
            \item $R(x,a,b)$.  
        \end{enumerate}
        We claim that $\Sigma$ is consistent modulo $T_{\mathbf{H}_4\text{-free}}$. Since $\Sigma$ extends $p$, it suffices to show that $4$-point sets containing $x$ and $a$ are $\mathbf{H}_4$-free. But this also follows immediately from the fact that $\Sigma$ extends $p$, using Remark \ref{rem:h4-free-criterion}. 

        \underline{Case 2}: $b' \neq a$.
        
        This time, let $\Sigma(x, b, b', a)$ be the extension of $p(x, b, \bar{c})$ given by adding the following formulas:
        \begin{enumerate}[(i)]
            \item $R(x, b', c_i) \leftrightarrow R(a, b, c_i)$ for all $i \in [n]$.
            \item $R(x, a, b) \wedge R(x, a, b')$.
            \item $R(x, a, c_i)$ for all $i \in [n]$.
            \item $R(x, b, b') \leftrightarrow R(a, b, b')$. 
        \end{enumerate}
        As in Case 1, we show that $\Sigma$ is consistent modulo $T_{\mathbf{H}_4\text{-free}}$ by showing that $4$-point sets containing $x$ and at least one of $a$ or $b'$ are $\mathbf{H}_4$-free. By (iii), $\Sigma$ implies that any set of the form $\{x,a,c_i,c_j\}$ is $\mathbf{H}_4$-free, and by (ii) and (iii), so is $\{x,a,b,c_i\}$. Moreover, by (i) and since $b \equiv_C b'$, $\Sigma$ implies that any set of the form $\{x, b', c_i, c_j\}$ or $\{x, b, b', c_i\}$ is $\mathbf{H}_4$-free. Finally, by (ii), $\Sigma$ implies that $\{x, a, b, b'\}$ is also $\mathbf{H}_4$-free, and (ii) and (iii) combined imply that so is $\{x, a, b', c_i\}$. 
        
        Therefore, in either case, $\Sigma$ defines a type over $a, b,b', \bar{c}$, and so we can choose $a' \models \Sigma$ in $M$. By (i) in both cases, this $a'$ works.    \hfill \pushQED{$\qed_{\text{Claim 1}}$}
    \end{proof}
    We now show (\ref{eq:aut-of-intersection}). Note that $\leq$ is clear. For $\geq$, let $g \in \Aut(M/C)$, and let $b' := g(b)$. Then, by Claim 1, there is $a' \equiv_{B} a$ such that $b \equiv_{Ca'} b'$. 
    \begin{nclaim}
        $\Aut(M/Ca') \leq H$.
    \end{nclaim}
    \begin{proof}[Proof of Claim 2]
        Since $a' \equiv_{B} a$, there is some $h \in \Aut(M/B)$ sending $a'$ to $a$. Hence, for any $k \in \Aut(M/Ca')$, we have that $hkh^{-1} \in \Aut(M/A) \leq H$. As $h \in \Aut(M/B) \leq H$, it follows that $k \in H$.\hfill \pushQED{$\qed_{\text{Claim 2}}$}
    \end{proof}
    Thus, since $b \equiv_{Ca'} b'$, there is some $k \in \Aut(M/Ca')$ such that $k(b) = b'$, that is, $k(b) = g(b)$. Therefore, $k^{-1}g \in H$, and so, since $k\in H$ by Claim 2, $g \in H$. 
\end{proof}
\begin{remark}
    Claim 1 in the previous proof is a strong property for which we need to restrict ourselves to both $a$ and $b$ being singletons. See Remark \ref{rem:counterexample-to-claim} for a counterexample to the version of the Claim obtained by replacing $b$ with a pair of elements.
\end{remark}
\begin{corollary}\label{cor:wei-acleq-is-acl}
    For all $\bar{a}, \bar{b}$ and $A$, $\bar{a} \equiv_{\acl^\eq(A)} \bar{b}$ iff $\bar{a} \equiv_{\acl(A)} \bar{b}$.
\end{corollary}
\begin{proof}
    This follows directly from \cite[Proposition 3.4]{casanovas2004wei}.
\end{proof}
\begin{proposition} \label{prop:not-i-extensible}
    $T_{\mathbf{H}_4\textnormal{-free}}$ is not $\ind^\textnormal{i}$-extensible.
\end{proposition}
\begin{proof}
    Note that, for any two pairs of distinct points $ab$ and $cd$, we have $ab \equiv cd$ in $T_{\mathbf{H}_4\textnormal{-free}}$. Since $T_{\mathbf{H}_4\textnormal{-free}}$ has trivial algebraic closure, we get $ab \equiv_{\acl(\varnothing)} cd$, and so, by Corollary \ref{cor:wei-acleq-is-acl}, it follows that $ab \equiv_{\acl^\eq(\varnothing)} cd$. Finally, since $T_{\mathbf{H}_4\text{-free}}$ is $\omega$-categorical, it follows that $ab \equiv^{\textnormal{Ls}} cd$. 
    
    Now assume, for contradiction, that there is some global $1$-type $p$ which is Lascar-invariant over $\varnothing$. Then, for any $a \neq b$, $ab \equiv^{\textnormal{Ls}} ba$, and thus, $R(x, a, b) \leftrightarrow R(x, b, a) \in p(x)$. Hence, $p$ is inconsistent modulo $T_{\mathbf{H}_4\textnormal{-free}}$, a contradiction. 
\end{proof}
\begin{remark}
    It follows that there cannot exist any canonical independence relation on the finite subsets of the countable homogeneous $\mathbf{H}_4$-free $3$-hypertournament in the sense of Kaplan and Simon (cf. \cite{kaplan2019automorphism}), since they all satisfy stationarity over $\varnothing$, and thus give rise to global invariant types. This contrasts with the situation for the countable homogeneous tournaments (see \cite[Example 4.4]{kaplan2019automorphism}).
\end{remark}
\begin{proposition}
    $T_{\mathbf{H}_4\textnormal{-free}}$ is locally $\ind^{\textnormal{i}}$-extensible.
\end{proposition}
\begin{proof}
    Let $p(\bar{x}) \in S(A)$ with $A \neq \varnothing$. Fix $\bar{a}^* \in A$. Let $q(\bar{x})$ be the global expansion of $p(\bar{x})$ obtained by adding the following formulas:
    \begin{enumerate}[(i)]
        \item $R(x_i, a, b)$ for all $a\in A$, $b\notin A$, and $i$.
        \item $R(x_i, b_1, b_2) \leftrightarrow R(a^*_i, b_1, b_2)$ for all $b_1, b_2 \notin A$ and $i$.
    \end{enumerate}
    We claim that this is consistent. Indeed, since $\Sigma$ extends $p$, any potential $\mathbf{H}_4$-structure must contain elements outside of $A$, and by (i), $q$ implies that any set of the form $\{x_i, a, a', b\}$, $\{x_i, a, b, b'\}$, or $\{x_i, x_j, a, b\}$ with $a, a' \in A$ and $b, b' \notin A$ is $\mathbf{H}_4$-free. Also, by (ii), $q$ implies that $4$-point sets of the form $\{x_i, b, b', b''\}$, $\{x_i, x_j, b, b'\}$, or $\{x_i, x_j, x_k, b\}$ with $b, b', b'' \notin A$ are $\mathbf{H}_4$-free. Therefore, $q$ is a global type.

    Furthermore, suppose that $b_1b_2 \equiv_A b'_1b'_2$. Since $\bar{a}^* \in A$, this means that $\models R(a^*_i, b_1, b_2) \leftrightarrow R(a^*_i, b'_1, b'_2)$ for all $i$, and thus, by (ii), $R(x_i, b_1, b_2) \leftrightarrow R(x_i, b'_1, b'_2) \in q$. Similarly, $\models R(a^*_i, a^*_j, b_1) \leftrightarrow R(a_i^*, a_j^*, b'_1)$ for all $i,j$, which implies by (ii) that $R(x_i, x_j, b_1) \leftrightarrow R(x_i, x_j, b'_1) \in q$. Therefore, $q$ is a global $A$-invariant extension of $p$, and so, by the argument from Proposition \ref{prop:not-i-extensible}, $q$ is Lascar-invariant over $A$.
\end{proof}
\begin{corollary} \label{cor:left-extension}
    In $T_{\mathbf{H}_4\textnormal{-free}}$, $\ind^\textnormal{i}$ satisfies left extension over non-empty bases.
\end{corollary}
\begin{proof}
    This is an immediate application of \cite[Lemma 3.23]{chernikov2012forking}.
\end{proof}
\section{First proof of \texorpdfstring{$\NSOP_4$}{NSOP4}} \label{sec:first-proof-of-nsop4}
In this section, we prove the main classification results for $T_{\mathbf{H}_4\textnormal{-free}}$ using concrete constructions instead of the more abstract machinery that will be used in the following sections. In what follows, we work in a monster model $\M \models T_{\mathbf{H}_4\textnormal{-free}}$, and as usual, assume that all sets and tuples are in $\M$ and all models elementarily embed in $\M$.
\begin{proposition}
    $T_{\mathbf{H}_4\textnormal{-free}}$ is $\NFOP_3$ and $\IP_2$. 
\end{proposition}
\begin{proof}
    $\NFOP_3$ follows directly from \cite[Proposition 3.23]{abd2023higher}. 

    We claim that the formula $R(x_1, x_2; y)$ has $\IP_2$. Fix a total order $<^*$ on $\mathcal{P}(\omega^2)$. Let $\Sigma(x_{0,i}, x_{1,j}, y_I)_{i, j \in \omega, I \subseteq \omega^2}$ be the following set of formulas:
    \begin{enumerate}[(i)]
        \item $R(x_{0,i}, x_{1,j}, y_I)$ iff $(i,j) \in I$.
        \item $R(x_{n,i}, x_{n,j}, x_{n,k})$ iff $i < j < k$, for any $n \in \{0,1\}$.
        \item $R(x_{n,i}, x_{n,j}, x_{n',l})$ iff $i<j$, for any $l$ and $\{n, n'\} = \{0,1\}$.
        \item $R(x_{n,i}, x_{n,j}, y_I)$ iff $i < j$, for any $I \subseteq \omega^2$ and $n \in \{0,1\}$.
        \item $R(x_{n,i}, y_I, y_J)$ iff $I <^* J$, for any $i$ and $n \in \{0,1\}$.
        \item $R(y_I, y_J, y_K)$ iff $I <^* J <^* K$.
    \end{enumerate}
    We claim that this is a complete type. The only thing that needs to be checked is that it does not generate any $\mathbf{H}_4$-structures. By (i) and (ii), $\Sigma$ implies that any $4$-point set containing only elements from $\bar{x}_0$ and $\bar{x}_1$ is $\mathbf{H}_4$-free, and, by (vi), it also implies that those containing only $y_I$'s are $\mathbf{H}_4$-free. By (iv), $\Sigma \vdash R(x_{n,i}, x_{n,j}, y_I) \leftrightarrow R(x_{n,i}, x_{n,j}, y_J)$ for all $I, J$, so that $\{x_{n,i}, x_{n,j}, y_I, y_J\}$ is $\mathbf{H}_4$-free, and by (v), $\Sigma \vdash R(x_{n,i}, y_I, y_J) \leftrightarrow R(x_{n', j}, y_I, y_J)$ for $\{n,n'\} = \{0,1\}$, so that $\{x_{n,i}, x_{n',j}, y_I, y_J\}$ is $\mathbf{H}_4$-free. Moreover, also by (iv), $\Sigma$ implies that $\{x_{n,i}, x_{n,j}, x_{n',l}, y_I\}$ is $\mathbf{H}_4$-free. Finally, since $<^*$ is a total order, by (v) $\Sigma$ implies that $\{x_{n,i}, y_I, y_J, y_K\}$ is $\mathbf{H}_4$-free. 

    Since these are all possible $\mathbf{H}_4$-structures, it follows that $\Sigma$ is consistent, and thus, we can choose a realisation $(a_{0,i}, a_{1,j}, b_I)_{i,j \in \omega, I \subseteq \omega^2} \models \Sigma$. By (i), this realisation witnesses $\IP_2$ for the formula $R(x_1, x_2, y)$.
\end{proof}
\begin{proposition}\label{prop:h4-free-is-sop3}
    $T_{\mathbf{H}_4\textnormal{-free}}$ is $\SOP_3$.
\end{proposition}
\begin{proof}
    Fix some elements $a,b$ from the monster model. We prove that the formula $\phi(x,y) := R(x,y,a) \wedge R(x,b,y)$ witnesses $\SOP_3$. First, assume $\phi(x_0, x_1) \wedge \phi(x_1, x_2) \wedge \phi(x_2, x_0)$ is consistent. Let $(c_0, c_1, c_2)$ be a realisation. Then the substructure on $\{c_0, c_1, c_2, a\}$ implies that $\models R(c_0, c_1, c_2)$, since otherwise we obtain an $\mathbf{H}_4$-structure. Similarly, the substructure on $\{c_0, c_1, c_2, b\}$ implies that $\models R(c_0, c_2, c_1)$, so we get a contradiction. Hence, $\phi(x_0, x_1) \wedge \phi(x_1, x_2) \wedge \phi(x_2,x_0)$ is inconsistent. 

    Now let $\Sigma(x_i)_{i \in \omega}$ be the set containing the following formulas over $ab$:
    \begin{enumerate}[(i)]
        \item $R(x_i, x_j, a) \wedge R(x_i, b, x_j)$ for all $i < j$.
        \item $R(x_i, x_j, x_k)$ iff $i < j < k$.
        \item $R(x_i, b, a)$ for all $i$.
    \end{enumerate}
    By (ii), $\Sigma$ implies that any $4$-point set containing only $x_i$'s is $\mathbf{H}_4$-free, and by (i), so is any $4$-point set of the form $\{x_i, x_j, x_k, a\}$ or $\{x_i, x_j, x_k, b\}$. Thus, any potential $\mathbf{H}_4$-structure must be on a set of the form $\{x_i, x_j, a, b\}$. But $\Sigma$ implies that this is $\mathbf{H}_4$-free by (iii).

    Therefore, $\Sigma(x_i)_{i \in \omega}$ is consistent, and so we can choose a realisation $(c_i)_{i \in \omega} \models \Sigma$. By stipulation, $\models \phi(c_i, c_j)$ for all $i < j$.
\end{proof}
\begin{proposition}\label{prop:h4-free-is-tp2}
    $T_{\mathbf{H}_4\textnormal{-free}}$ is $\TP_2$.
\end{proposition}
\begin{proof}[Proof of Proposition \ref{prop:h4-free-is-tp2} (v.1)]
    For convenience, let us expand $\mathcal{L}$ by adding two constant symbols $e$ and $f$. Let $\bar{b} = (c,d)$ be a pair such that
    \begin{equation*}
        \models R(c, d, e) \wedge R(d, e, f) \wedge R(c, e, f) \wedge R(c, f, d).
    \end{equation*}
    Now let $\phi(x, \bar{b})$ be the following $\mathcal{L}(e,f)$-formula:
    \begin{equation*}
        R(x, c, d) \wedge R(x, d, e) \wedge R(x, e, f) \wedge R(x, c, f) \wedge R(x, e, c) \wedge R(x, f, d).
    \end{equation*}
    It is easy to see from the definition that, for any $a \models \phi(x, \bar{b})$, the set $\{a, c, d, e, f\}$ is $\mathbf{H}_4$-free, and therefore, $\phi(x, \bar{b})$ isolates a complete type in $\mathcal{L}(e,f)$.

    Now let $\Sigma(\bar{v}_i)_{i \in \omega}$ with $\bar{v}_i := (x_i, y_i)$ be the following set of formulas:
\begin{enumerate}[(i)]
    \item $\bar{v}_i \equiv_{ef} \bar{b}$ for all $i$.
    \item $R(e, y_j, x_i) \wedge R(f, x_i, y_j)$ iff $i < j$.
    \item $R(e, x_i, x_j) \wedge R(f, x_i, x_j) \wedge R(e, y_i, y_j) \wedge R(f, y_i, y_j)$ iff $i < j$.
    \item $R(x_i, x_j, x_k) \wedge R(y_i, y_j, y_k)$ iff $i < j < k$.
    \item $R(x_i, x_j, y_n)$ iff $i<j$ and $R(x_m, y_k, y_l)$ iff $k<l$, for all $m$ and $n$.
\end{enumerate}
We claim that $\Sigma(\bar{v}_i)_{i \in \omega}$ is consistent. Indeed, by (i), it suffices to consider those $4$-sets containing elements from different $\bar{v}_i$'s. It is straightforward using (iv) and (v) that $4$-sets containing only $x_i$'s and $y_j$'s are $\mathbf{H}_4$-free. Similarly, by (ii) and (iii), $4$-sets containing only $x_i$'s, $y_j$'s and exactly one of $e$ or $f$ are also $\mathbf{H}_4$-free, and similarly for those $4$-sets containing both $e$ and $f$ and elements from exactly one of $\bar{x}$ or $\bar{y}$. So we just need to check that sets of the form $\{e, f, x_i, y_j\}$ are $\mathbf{H}_4$-free. But these are $\mathbf{H}_4$-free by (i). Note that a realisation will give an $ef$-indiscernible sequence.

Hence, $\Sigma(\bar{v}_i)_{i \in \omega}$ defines a type in $\mathcal{L}(e,f)$, so we can choose a realisation $\bar{b}_i := (c_i, d_i)_{i \in \omega} \models \Sigma$. Then, by (ii), the substructure on $\{x,d_j,e,c_i\}$ implies that $\phi(x, \bar{b}_i) \wedge \phi(x, \bar{b}_j) \vdash R(x, d_j, c_i)$ for $i < j$, since otherwise we obtain an $\mathbf{H}_4$-structure. Similarly, the substructure on $\{x,c_i,f,d_j\}$ implies that $\phi(x, \bar{b}_i) \wedge \phi(x, \bar{b}_j) \vdash R(x, c_i, d_j)$ for $i < j$. Therefore, the set $\{\phi(x, \bar{b}_i) : i \in \omega\}$ is $2$-inconsistent.

To make this into an instance of $\TP_2$, we just take copies of this indiscernible sequence. Formally, we choose an array $(\bar{b}_{i,j})_{i, j \in \omega} = (c_{i,j},d_{i,j})_{i, j \in \omega}$ such that:
\begin{enumerate}
    \item[(i$'$)] For all $i \in \omega$, $(\bar{b}_{i,j})_{j \in \omega} \equiv_{ef} (\bar{b}_j)_{j \in \omega}$.
    \item[(ii$'$)] For any $i \neq j$, $l$ and $m$, we have
        \begin{align*}
            \models R(z, c_{j,m}, d_{j,m}) \wedge R(z, c_{j,m}, e) \wedge R(z, c_{j,m}, f) \wedge R(z, d_{j,m}, e) \wedge R(z, d_{j,m}, f)
        \end{align*}
        for any $z \in \bar{b}_{i,l}$.
    \item[(iii$'$)] For any $i < j < k$, any $l,m,n$, and $z_\ell \in \bar{b}_{\ell,u}$ for $\ell \in \{i,j,k\}$ and $u \in \{l,m,n\}$, we have $\models R(z_i,z_j,z_k)$. 
\end{enumerate}
It is then easy to see that (i$'$) implies that $\{\phi(x, \bar{b}_{i,j}) : j \in \omega\}$ is $2$-inconsistent for each $i \in \omega$ by the argument from the previous paragraph, and that (ii$'$) and (iii$'$) jointly imply that $\{\phi(x, \bar{b}_{i, f(i)}) : i \in \omega\}$ is consistent for any $f \colon \omega \to \omega$. 
\end{proof}
\begin{remark}\label{rem:counterexample-to-claim}
    Let us note, using the choice of elements from the previous proof, that the inconsistency of $\phi(x, \bar{b}_0) \wedge \phi(x,\bar{b}_1)$ implies that, for any $a \models \phi(x, \bar{b}_0)$, we cannot find any $a' \equiv_{ef\bar{b}_0} a$ such that $\bar{b}_1 \equiv_{efa'} \bar{b}_0$. This shows that we cannot extend Claim 1 from the proof of Lemma \ref{lem:wei} to tuples of finite length greater than 1. 
\end{remark}
\begin{proposition}\label{prop:h4-free-is-nsop4}
    $T_{\mathbf{H}_4\textnormal{-free}}$ is $\NSOP_4$.
\end{proposition}
\begin{proof}[Proof of Proposition \ref{prop:h4-free-is-nsop4} (v.1)]
    It suffices to show that, for any indiscernible sequence $(\bar{b}_i)_{i < \omega}$ with common intersection $\bar{c}$ (possibly empty), and letting $p(\bar{x}, \bar{y}) := \tp(\bar{b}_0 \bar{b}_1)$, $p(\bar{x}_0, \bar{x}_1) \wedge p(\bar{x}_1, \bar{x}_2) \wedge p(\bar{x}_2, \bar{x}_3) \wedge p(\bar{x}_3, \bar{x}_0)$ is consistent. By indiscernibility, we can equivalently take $(\bar{b}_i)_{i < \omega}$ to be $\bar{c}$-indiscernible and consist only of those elements outside of $\bar{c}$ from each tuple. 

    So let $(\bar{b}_i)_{i < \omega}$ be a $\bar{c}$-indiscernible sequence as above, and let $p(\bar{x}, \bar{y}) := \tp(\bar{b}_0\bar{b}_1/\bar{c})$. Let $n := \size{\bar{b}_0}$. Notice that $p$ is of the following form:
    \begin{equation*}
        p(\bar{x}, \bar{y}) = \tp(\bar{x}/\bar{c}) \wedge \tp(\bar{y}/\bar{c}) \wedge \tp(\bar{x}\bar{y}) \wedge \bigwedge_{j,k<n} R(c_i, x_j, y_k)^{\epsilon(i,j,k)},
    \end{equation*}
    where $\epsilon(i,j,k) \in \{0,1\}$. (Note that $j,k$ can be equal.)

    Define $\Sigma(\bar{x})$ with $\size{\bar{x}}=n$ as the set containing the following formulas:
    \begin{enumerate}[(i)]
        \item $\bar{b}_0 \bar{b}_1 \equiv_{\bar{c}} \bar{b}_1 \bar{x}$.
        \item $R(c_i, b_{0,j}, x_k)$ for all $i$ and $j, k < n$ (not necessarily distinct).
        \item $R(b_{0,i}, x_j, b_{1,k})$ for all $i, j, k < n$.
    \end{enumerate}
    We claim that this is $\mathbf{H}_4$-free. Indeed, by (i), any potential $\mathbf{H}_4$-structures involving $\bar{b}_0$, $\bar{b}_1$ and $\bar{x}$ arise from the last big conjunction in $p(\bar{b}_0, \bar{b}_1)$ and $p(\bar{b}_1, \bar{x})$. It also follows from (i) and (iii) that there cannot be an $\mathbf{H}_4$-structure which does not contain any element from $\bar{c}$. Thus, it suffices to consider possible $\mathbf{H}_4$-structures on sets of the form $\{c_i, b_{0,j}, b_{1,k}, x_l\}$ for some fixed $i,j,k$ and $l$. 
    \begin{itemize}
        \item If $\epsilon(i,j,k) = \epsilon(i,k,l) = 1$, then $\models R(c_i, b_{0,j}, b_{1,k})$ and $\Sigma(\bar{x}) \vdash R(c_i, b_{0,j}, x_l)$ by (ii). Therefore, $\Sigma(\bar{x})$ implies that $\{c_i, b_{0,j}, b_{1,k}, x_l\}$ is $\mathbf{H}_4$-free.
        \item If $\epsilon(i,j,k) = \epsilon(i,k,l) = 0$, then $\models R(b_{0,j}, c_i, b_{1,k})$ and $\Sigma(\bar{x}) \vdash R(b_{0,j}, x_l, b_{1,k})$ by (iii). Hence, once again, $\Sigma(\bar{x})$ implies that $\{c_i, b_{0,j}, b_{1,k}, x_l\}$ is $\mathbf{H}_4$-free.
        \item If $\epsilon(i,j,k) = 1$ but $\epsilon(i,k,l) = 0$, then we have $\models R(c_i, b_{0,j}, b_{1,k})$ and $\Sigma(\bar{x}) \vdash R(c_i, x_l, b_{1,k})$ by (i), and so $\Sigma(\bar{x})$ implies that $\{c_i, b_{0,j}, b_{1,k}, x_l\}$ is $\mathbf{H}_4$-free. A similar argument works if $\epsilon(i,j,k)=0$ but $\epsilon(i,k,l) = 1$.
    \end{itemize}
    So $\Sigma(\bar{x})$ defines a partial type over $\bar{b}_0\bar{b}_1\bar{c}$ (note that we do not specify $\tp(\bar{b}_0\bar{x})$ since it is not needed for our purposes), and so we choose a realisation $\bar{b}_2^* \models \Sigma$.

    Now define $\Gamma(\bar{y})$ with $\size{\bar{y}}=n$ as the set containing the following formulas:
    \begin{enumerate}
        \item[(i$'$)] $\bar{y} \bar{b}_0 \equiv_{\bar{c}} \bar{b}_2^* \bar{y} \equiv_{\bar{c}} \bar{b}_0 \bar{b}_1$.
        \item[(ii$'$)] $R(b_{0,i}, b_{2,j}^*, y_k)$ for all $i, j, k < n$.
    \end{enumerate}
    Once again, by the two stipulations, it suffices to show that no set of the form $\{c_i, b^*_{2,j}, y_k, b_{0,l}\}$ gives rise to an $\mathbf{H}_4$-structure for some fixed $i,j,k$ and $l$. 
    \begin{itemize}
        \item If $\epsilon(i,j,k) = \epsilon(i,k,l) = 1$, then $\Gamma(\bar{y}) \vdash R(c_i, b_{2,j}^*, y_k) \wedge R(b_{0,l}, b_{2,j}^*, y_k)$, and thus $\Gamma(\bar{y})$ implies $\{c_i, b_{2,j}^*, y_k, b_{0,l}\}$ is $\mathbf{H}_4$-free.
        \item If $\epsilon(i,j,k) = \epsilon(i,k,l) = 0$, then we have $\models R(c_i, b_{0,l}, b_{2,j}^*)$ and $\Gamma(\bar{y}) \vdash R(c_i, b_{0,l}, y_k)$. Hence, once again, $\Gamma(\bar{y})$ implies that $\{c_i, b_{2,j}^*, y_k, b_{0,l}\}$ is $\mathbf{H}_4$-free.
        \item If $\epsilon(i,j,k) = 1$ but $\epsilon(i,k,l) = 0$, then $\Gamma(\bar{y}) \vdash R(c_i, b_{2,j}^*, y_k) \wedge R(c_i, b_{0,l}, y_k)$, and so $\Gamma(\bar{y})$ implies that $\{c_i, b_{2,j}^*, y_k, b_{0,l}\}$ is $\mathbf{H}_4$-free. A similar argument works if $\epsilon(i,j,k) = 0$ but $\epsilon(i,k,l) = 1$.
    \end{itemize}
    Hence, again, $\Gamma(\bar{y})$ defines a type over $\bar{b}_0\bar{b}_2^*\bar{c}$, and so by $\omega$-saturation we can choose a realisation $\bar{b}_3^* \models \Gamma$. Then $\models p(\bar{b}_0, \bar{b}_1) \wedge p(\bar{b}_1, \bar{b}_2^*) \wedge p(\bar{b}_2^*, \bar{b}_3^*) \wedge p(\bar{b}_3^*, \bar{b}_0).$
\end{proof}
This completes our first proof of Theorem \ref{thm:main}.
\begin{remark}
    Surprisingly, we have not been able to produce a direct proof that the $\mathbf{H}_4$-free $3$-hypertournament is $\NSOP$, despite the fact that $\NSOP_4$ is a stronger property. 
\end{remark}
\section{\texorpdfstring{$\leq_K$}{leqK}-independence} \label{sec:lascar-independence}
In order to apply Mutchnik's framework from \cite{mutchnik2022conantindependence} to offer a different proof of the above classification results for the $\mathbf{H}_4$-free $3$-hypertournament, we need to slightly generalise them. In their current form, the results from that paper apply whenever we have an independence relation satisfying, among other properties, stationarity over models. However, as will become clear in the next section, we want to apply them in the context of a non-stationary independence relation. To achieve this, in this section, we introduce an independence relation having a close relationship with \textit{strong Lascar independence} as defined in \cite{tartarotti2023lascar}.
\subsection{Defining \texorpdfstring{$\leq_K$}{leqK}-independence}
In this section, we work in a general theory $T$ with monster model $\M$.
\begin{definition}
    For $M \models T$, we define $\bar{a} \ind^{\leq_K}_M \bar{b}$ if $\tp(\bar{a}/M\bar{b})$ extends to a $\leq_K$-least extension of $\tp(\bar{a}/M)$.
\end{definition}
We introduce some auxiliary notation that will be useful for the following proofs. Let $\sigma \in \Aut(\M)$. For a formula $\phi(\bar{x}) \in \mathcal{L}(M)$, let us write $\sigma(\phi(\bar{x})) \in \mathcal{L}(\sigma(M))$ for the formula where the parameters are shifted by $\sigma$. Similarly, for a global type $q$, we write $\sigma(q) := \{\sigma(\phi(\bar{x})) : \phi(\bar{x}) \in q\}$.
\begin{lemma}\label{lem:invariance-for-lascar}
    Let $q$ be an global $M$-invariant type, and let $\sigma \in \Aut(\M)$. Then $\phi(\bar{x},\bar{y}) \in \cl_K(q)$ iff $\sigma(\phi(\bar{x},\bar{y})) \in \cl_K(\sigma(q))$. 
\end{lemma}
\begin{proof}
    It suffices to show one of the implications. First note that, as $q$ is $M$-invariant, $\sigma(q)$ is $\sigma(M)$-invariant. Now suppose $\phi(\bar{x},\bar{y}) \in \cl_K(q)$. Let $(\bar{b}_i)_{i \in \omega}$ be a Morley sequence in $\sigma(q)$ over $\sigma(M)$. Then, for any $i$, since $\bar{b}_i \models \sigma(q)|_{\sigma(M)\bar{b}_{<i}}$ and $\sigma$ is an automorphism, it follows that $\sigma^{-1}(\bar{b}_i) \models q|_{M\sigma^{-1}(\bar{b}_{<i})}$. Therefore, $(\sigma^{-1}(\bar{b}_i))_{i \in \omega}$ is a Morley sequence in $q$ over $M$, and thus, by assumption, $\{\phi(\bar{x}, \sigma^{-1}(\bar{b}_i)) : i \in \omega\}$ is inconsistent. Hence, $\{\sigma(\phi(\bar{x}, \bar{b}_i)) : i \in \omega\}$ is inconsistent, as required. 
\end{proof}
\begin{corollary}
    $\ind^{\leq_K}$ is an independence relation (i.e., it is $\Aut(\M)$-invariant). 
\end{corollary}
\begin{proof}
    By Lemma \ref{lem:invariance-for-lascar}, we have that $q$ is a $\leq_K$-least extension of $q|_M$ iff $\sigma(q)$ is a $\leq_K$-least extension of $\sigma(q)|_{\sigma(M)}$ for any $\sigma \in \Aut(\M)$.
\end{proof}
\begin{remark} \label{rem:lascar-basic-props}
    It is immediate from the definition that $\leq_K$-independence satisfies right monotonicity and right extension. In particular, it has existence iff it has full existence. $M$-$\ind^{\leq_K}$-free types are precisely the global types which are $\leq_K$-least over their restrictions to $M$.
\end{remark}
\begin{lemma}\label{lem:lascar-left-mon}
    If $\ind^{\textnormal{i}}$ satisfies left extension over models, then $\ind^{\leq_K}$ satisfies left monotonicity. 
\end{lemma}
\begin{proof}
    Given a type $p(\bar{x}, \bar{y}) \in S(M)$ and a global $M$-invariant extension $q(\bar{x}, \bar{y})$, let us write $\Tilde{p}(\bar{y})$ and $\Tilde{q}(\bar{y})$ for their respective restrictions to complete types in the variables $\bar{y}$. It suffices to show that, if $q(\bar{x},\bar{y})$ is a $\leq_K$-least extension of $p(\bar{x},\bar{y}) \in S(M)$, then $\Tilde{q}(\bar{y})$ is a $\leq_K$-least extension of $\Tilde{p}(\bar{y})$. 
    
    Suppose $\phi(\bar{z},\bar{y}) \in \cl_K(\Tilde{q}(\bar{y}))$. Then $\phi(\bar{z}, \bar{y}) \in \cl_K (q(\bar{x}, \bar{y}))$, and so, since $q(\bar{x}, \bar{y})$ is a $\leq_K$-least extension of $p(\bar{x}, \bar{y})$, it follows that $\phi(\bar{z},\bar{y}) \in \cl_K(r(\bar{x},\bar{y}))$ for all global $M$-invariant extensions $r(\bar{x},\bar{y})$ of $p(\bar{x},\bar{y})$. Now, for any global $M$-invariant extension $\Tilde{r}(\bar{y})$ of $\Tilde{p}(\bar{y})$, we can find, by left extension for $\ind^{\text{i}}$, a global $M$-invariant completion of $p(\bar{x},\bar{y}) \cup \Tilde{r}(\bar{y})$, and thus it follows that $\phi(\bar{z},\bar{y}) \in \cl_K(\Tilde{r}(\bar{y}))$. So $\Tilde{q}(\bar{y})$ is a $\leq_K$-least extension of $\Tilde{p}(\bar{y})$.
\end{proof}
Beyond satisfying the relative Kim's lemma, $\leq_K$-independence also provides a criterion for this property:
\begin{proposition} \label{prop:relative-kims-lemma-lascar-criterion}
    Suppose $\ind$ is an independence relation satisfying full existence. Then $\ind$ satisfies the relative Kim's lemma iff $\ind \implies \ind^{\leq_K}$.
\end{proposition}
\begin{proof}
    ($\Rightarrow$) Suppose $\bar{a} \ind_M \bar{b}$. By full existence, $\tp(\bar{a}/M\bar{b})$ extends to a global $M$-$\ind$-free type $q$ which, by the relative Kim's lemma, is $\leq_K$-least. So $\bar{a} \ind_M^{\leq_K} \bar{b}$.

    ($\Leftarrow$) Suppose $q$ is a global type that is $\ind$-free over $M \models T$. Since $\ind \implies \ind^{\leq_K}$, it follows that $q$ is also $\ind^{\leq_K}$-free over $M$, and thus, by Remark \ref{rem:lascar-basic-props}, $q$ is a $\leq_K$-least extension of $q|_M$. 
\end{proof}
\subsection{A criterion for the strong witnessing property}
The following property was introduced by Mutchnik in \cite[Definition 3.23]{mutchnik2022conantindependence}:
\begin{definition}
    Let $T$ be a complete theory. 
    \begin{enumerate}[(i)]
        \item We say a global type $q$ is a \textbf{strong witnessing extension} of $q|_M$ for $M \models T$ if, for all $\bar{a} \models q|_{M\bar{b}}$, and all $\bar{c}$, there is some $\bar{c}' \equiv_{M\bar{a}} \bar{c}$ such that $\tp(\bar{a} \bar{c}'/M\bar{b})$ extends to a $\leq_K$-least extension of $\tp(\bar{a}\bar{c}/M)$.
        \item We say $T$ has the \textbf{strong witnessing property} if, for all small $M \models T$, every $p \in S(M)$ has a global strong witnessing extension. 
    \end{enumerate}
\end{definition}
There are no known $\NSOP_4$ counterexamples to the above property. The main application is given in \cite[Theorem 3.25]{mutchnik2022conantindependence}, where it is shown that any theory with the strong witnessing property is either simple or $\TP_2$.

Let us introduce the following operation: given an independence relation $\ind$, we define
\begin{equation*}
    A \ind_C^{\textnormal{le}} B \iff \text{for all } D \supseteq A, \text{ there exists } B' \equiv_{AC} B \text{ with } D \ind_C B'.
\end{equation*}
Some versions of this operation have previously appeared in the literature, although not at this level of generality (e.g., \cite{dobrowolski2022kim}). It is immediate from this definition that:
\begin{itemize}
    \item $\ind^{\text{le}} \implies \ind$.
    \item If $\ind^0 \implies \ind$ and $\ind^0$ satisfies left extension, then $\ind^0 \implies \ind^{\text{le}}$.
    \item $\ind^{\textnormal{le}} = \Big(\Big(\ind^{\textnormal{opp}}\Big)^*\Big)^{\textnormal{opp}}$.
\end{itemize}
In particular, combining the above properties with Fact \ref{fact:ind-star-props} and Remark \ref{rem:ind-opp-props}, it follows that, whenever $\ind$ satisfies left monotonicity, $\ind^{\textnormal{le}}$ satisfies left monotonicity and left extension, and if in addition $\ind$ satisfies right monotonicity, so does $\ind^{\textnormal{le}}$. 
\begin{lemma} \label{lem:strong-witnessing-is-lind-le}
    Let $M \models T$, and let $q$ be a global type. The following are equivalent:
    \begin{enumerate}[(i)]
        \item $q$ is a strong witnessing extension of $q|_M$.
        \item $q$ is $(\ind^{\leq_K})^{\textnormal{le}}$-free over $M$.
    \end{enumerate}
\end{lemma}
\begin{proof}
    ((i) $\Rightarrow$ (ii)) Suppose $q$ is a strong witnessing extension of $q|_M$. Let $\bar{a} \models q|_{M\bar{b}}$. Take some $\bar{c}$. Then, by strong witnessing, there is $\bar{c}' \equiv_{M\bar{a}} \bar{c}$ such that $\bar{a}\bar{c}' \ind_M^{\leq_K} \bar{b}$, and thus, by invariance, there is $\bar{b}' \equiv_{M\bar{a}} \bar{b}$ such that $\bar{a}\bar{c} \ind_M^{\leq_K} \bar{b}'$. Therefore, we have $\bar{a} \: (\ind_M^{\leq_K})^{\text{le}} \: \bar{b}$, as required.
    
    ((ii) $\Rightarrow$ (i)) Suppose $q$ is $(\ind^{\leq_K})^{\text{le}}$-free over $M$. Take $\bar{a} \models q|_{M\bar{b}}$, and some $\bar{c}$. Then, by assumption and invariance, there is $\bar{c}' \equiv_{M\bar{a}} \bar{c}$ such that $\bar{a}\bar{c}' \ind^{\leq_K}_M \bar{b}$. Therefore, $q$ is a strong witnessing extension of $q|_M$, as required.
\end{proof}
\begin{theorem}
    If there is an independence relation $\ind$ over models of $T$ satisfying full existence, left extension, and the relative Kim's lemma, then $T$ satisfies the strong witnessing property. 
    
    Moreover, if, in addition, $\ind^{\textnormal{i}}$ satisfies left extension over models in $T$, then the converse holds. 
\end{theorem}
\begin{proof}
    Suppose that such an independence relation $\ind$ exists. By Proposition \ref{prop:relative-kims-lemma-lascar-criterion}, we have $\ind \implies \ind^{\leq_K}$, and since $\ind$ satisfies left extension, we conclude that $\ind \implies (\ind^{\leq_K})^\text{le}$. But now, since $\ind$ satisfies full existence, every type $p \in S(M)$ has a global $M$-$\ind$-free extension, which by the above is also $(\ind^{\leq_K})^{\text{le}}$-free over $M$. Thus, by Lemma \ref{lem:strong-witnessing-is-lind-le}, $T$ satisfies the strong witnessing property. 

    For the `moreover' part, suppose that $T$ satisfies the strong witnessing property. In particular, by Lemma \ref{lem:strong-witnessing-is-lind-le}, this entails that every type $p \in S(M)$ has a global $M$-$(\ind^{\leq_K})^{\text{le}}$-free extension, which implies that $(\ind^{\leq_K})^\text{le}$ satisfies full existence. In particular, by Proposition \ref{prop:relative-kims-lemma-lascar-criterion}, $(\ind^{\leq_K})^{\text{le}}$ satisfies the relative Kim's lemma. Finally, as $\ind^{\text{i}}$ satisfies left extension by assumption, $\ind^{\leq_K}$ satisfies left monotonicity by Lemma \ref{lem:lascar-left-mon}, and thus $(\ind^{\leq_K})^{\text{le}}$ satisfies left extension. 
\end{proof}
\begin{corollary} \label{cor:simple-or-tp2}
    Suppose there is an independence relation $\ind$ over models satisfying full existence, left extension, and the relative Kim's lemma. Then $T$ is either simple or $\TP_2$.
\end{corollary}
\begin{proof}
    This is just a reformulation of \cite[Theorem 3.25]{mutchnik2022conantindependence}.
\end{proof}
\section{Second proof of \texorpdfstring{$\NSOP_4$}{NSOP4}} \label{sec:second-proof-of-nsop4}
\subsection{Triviality of Conant-independence} Our second proof of Theorem \ref{thm:main} goes via properties of independence relations, using several results from \cite{mutchnik2022conantindependence}. First, we recall the key result we employ in this section:
\begin{definition}
    Let $T$ be a complete theory and $M \models T$. 
    \begin{enumerate}[(i)]
        \item We say $\phi(\bar{x}, \bar{b})$ \textbf{Conant-divides over $M$} if, for all $M$-invariant Morley sequences $(\bar{b}_i)_{i \in \omega}$ with $\bar{b}_0 = \bar{b}$, the set $\{\phi(\bar{x}, \bar{b}_i) : i \in \omega\}$ is inconsistent. 
        \item We say $\phi(x, b)$ \textbf{Conant-forks over $M$} if there exist formulas $\psi_i(x, c_i)$ for $i < n$ such that $\phi(x, b) \vdash \bigvee_{i < n} \psi_i(x, c_i)$ and each $\psi_i(x, c_i)$ Conant-divides over $M$.
        \item We say $\bar{a}$ is \textbf{Conant-independent} from $\bar{b}$ over $M$ if $\tp(\bar{a}/M\bar{b})$ does not contain a formula that Conant-forks over $M$.
    \end{enumerate}
\end{definition}
\begin{remark} \label{rem:h4-conant-fork-eq-conant-div}
    In $T_{\mathbf{H}_4\text{-free}}$, since $\ind^{\textnormal{i}}$ satisfies left extension over models by Corollary \ref{cor:left-extension}, we have that $\phi(\bar{x}, \bar{b})$ Conant-forks over $M$ iff it Conant-divides over $M$. The proof of this fact is exactly the same as that of \cite[Proposition 5.2]{mutchnik2023nsop2} after replacing all instances of ``coheir Morley sequence over $M$'' by ``$M$-invariant Morley sequence,'' since left extension is the main property of coheir extensions used there. (Note that Conant-dividing has a different definition in \cite{mutchnik2023nsop2} than the one we are using in this document.)
\end{remark}
\begin{fact}[\protect{\cite[Theorem 6.2]{mutchnik2022conantindependence}}] \label{fact:symmetric-conant-independence}
    If Conant-independence is symmetric, then $T$ is $\NSOP_4$.
\end{fact}
Now work in $\M \models T_{\mathbf{H}_4\text{-free}}$. Let us define an independence relation $\ind^\hyp$ over subsets of $\M$ as follows: $A \ind^\hyp_C B$ iff $A \cap B \subseteq C$ and $\models R(a, c, b)$ for all $a \in A \setminus C$, $c \in C$, and $b \in B \setminus C$.
\begin{lemma}
    $\ind^\hyp$ is an independence relation satisfying monotonicity, full existence, and left extension. 
\end{lemma}
\begin{proof}
    Invariance and monotonicity are immediate. Full existence and left extension follow similar proofs to those of consistency we have done before. 
\end{proof}
\begin{remark}
    $\ind^\hyp$ does not satisfy stationarity (over arbitrary subsets) nor symmetry. In fact, if $A \ind^\hyp_C B$ and $\varnothing \neq C \subset A,B$, then, for all $a \in A\setminus C$, $c \in C$, and $b \in B \setminus C$, $\models R(a, c, b)$, and thus, $\not\models R(b, c, a)$. Hence, $B \nind^\hyp_C A$. 

    So $\ind^\hyp$ is neither a stationary weak independence relation in the sense of \cite{li2019automorphism} nor a free amalgamation relation in the sense of \cite{conant2017freeamalgamation}. It is also different from the independence relations studied in \cite{mutchnik2022conantindependence}.
\end{remark}
\begin{remark}
    We can characterise $\ind^\hyp$-Morley sequences over some $M \models T_{\mathbf{H}_4\textnormal{-free}}$ starting at some $\bar{b}_0$ disjoint from $M$ as precisely those sequences $(\bar{b}_i)_{i \in \omega}$ such that $\models R(m, b_{i,k}, b_{j,l})$ for all $i < j$ and any indices $k, l$ and elements $m \in M$.
\end{remark}
\begin{proposition} \label{prop:indht-kim-ind-is-trivial}
    For (possibly infinite) tuples $\bar{a}$ and $\bar{b}$, and a model $M \models T_{\mathbf{H}_4\textnormal{-free}}$, $\tp(\bar{a}/M\bar{b})$ does not $\ind^\hyp$-Kim-divide over $M$ iff $\bar{a} \cap \bar{b} \subseteq M$. 
\end{proposition}
\begin{proof}
    It is enough to prove the right-to-left direction. Suppose $\bar{a} \cap \bar{b} \subseteq M$. Note that we may assume that $\bar{a}$ and $\bar{b}$ are both finite tuples. Also note that, if either $\bar{a} \subseteq \bar{b}$ or $\bar{b} \subseteq \bar{a}$ hold, then there is nothing to do. So, without loss of generality, suppose that $\bar{a} \cap \bar{b} = \varnothing$ (in particular, they are disjoint from $M$). Let $(\bar{b}_i)_{i \in \omega}$ be an $\ind^\hyp$-Morley sequence over $M$ starting with $\bar{b}$. Let $\phi(\bar{x},\bar{b}) \in \tp(\bar{a}/M\bar{b})$, which, since $T_{\mathbf{H}_4\textnormal{-free}}$ is $\omega$-categorical, we may assume is a formula isolating a complete type over $M'\bar{b}$ for some finite $M' \subset M$. We claim that $\{\phi(\bar{x}, \bar{b}_i) : i \in \omega\}$ is consistent. 

    We show that
    \begin{equation*}
        \{\phi(\bar{x}, \bar{b}_i) : i \in \omega\} \cup \{R(x_k, b_{i,l}, b_{j,l'}) : i < j \in \omega\}
    \end{equation*}
    is consistent. Note that, by the condition on elements from different indexed tuples from the $\ind^\hyp$-Morley sequence and the fact that $\phi$ isolates a complete type, all $4$-substructures containing only variables and elements from the $\bar{b}_i$'s are $\mathbf{H}_4$-free. Moreover, since $\phi$ isolates a type, we have, for any $i, j \in \omega$, that $R(x_k, m, b_{i,l}) \in \phi(\bar{x}, \bar{b}_i)$ iff $R(x_k, m, b_{j,l}) \in \phi(\bar{x}, \bar{b}_j)$, and thus the set $\{x_k, m, b_{i,l}, b_{j,l}\}$ is $\mathbf{H}_4$-free. Finally, since $(\bar{b}_i)_{i \in \omega}$ is an $\ind^\hyp$-Morley sequence over $M$, we have $\models R(m, b_{i,l}, b_{j,l'})$ whenever $i < j$, and so it follows that $\{x, m, b_{i,l}, b_{j,l'}\}$ is $\mathbf{H}_4$-free. These are all the possible $\mathbf{H}_4$-substructures.
    
    Therefore, $\tp(\bar{a}/M\bar{b})$ does not $\ind^\hyp$-Kim-divide over $M$, as required. 
\end{proof}
\begin{lemma} \label{lem:ind-hti-morley-seqs-exist}
    Every type $p \in S(M)$ has a global $M$-invariant extension that is $\ind^\hyp$-free over $M$. 
\end{lemma}
\begin{proof}
    It suffices to show that $p$ has a global $M$-invariant extension $q$ such that $R(x, m, a) \in q$ for all $m \in M$ and $a \notin M$. Let us write $p(\bar{x}) = \tp(\bar{b}/M)$ for some realisation $\bar{b}$ in the monster. We may assume, without loss of generality, that $\bar{b}$ is disjoint from $M$. Fix some tuple $\bar{m}$ from $M$ of the same length as $\bar{b}$. Then it is enough to show that the extension of $\tp(\bar{b}/M)$ obtained by adding the formulas:
    \begin{enumerate}[(i)]
        \item $R(x, m, a)$ for all $m \in M$ and $a \notin M$;
        \item $\phi(\bar{x}, \bar{a})$ whenever $\bar{a} \cap M = \varnothing$ and $\models \phi(\bar{m}, \bar{a})$
    \end{enumerate}
    is consistent. (Note that $M$-invariance is clear.)

    As usual, it suffices to check all possible $\mathbf{H}_4$-substructures. Note that, by (ii), no $4$-point set containing only variables and elements outside of $M$ can be $\mathbf{H}_4$. So the only possibilities are $\{x_i, x_j, m, a\}$, $\{x_i, m, m', a\}$ and $\{x_i, m, a, a'\}$ for $m, m' \in M$ and $a, a' \notin M$. But all these are $\mathbf{H}_4$-free by (i).
\end{proof}
\begin{proof}[Proof of Proposition \ref{prop:h4-free-is-nsop4} (v.2)]
    We claim that $\bar{a}$ is Conant-independent from $\bar{b}$ over $M$ iff $\bar{a} \cap \bar{b} \subseteq M$. The left-to-right direction is always true. For the converse, suppose that $\bar{a} \cap \bar{b} \subseteq M$. By Proposition \ref{prop:indht-kim-ind-is-trivial}, $\tp(\bar{a}/M\bar{b})$ does not $\ind^\hyp$-Kim-divide over $M$. Moreover, by Lemma \ref{lem:ind-hti-morley-seqs-exist}, there exists an $M$-invariant $\ind^\hyp$-Morley sequence $(\bar{b}_i)_{i \in \omega}$ over $M$ with $\bar{b}_0 = b$. So, for all $\phi(\bar{x}, \bar{b}) \in \tp(\bar{a}/M\bar{b})$, $\{\phi(\bar{x}, \bar{b}_i) : i \in \omega\}$ is consistent, and thus, $\phi(\bar{x}, \bar{b})$ does not Conant-divide over $M$. Therefore, by Remark \ref{rem:h4-conant-fork-eq-conant-div}, $\bar{a}$ is Conant-independent from $\bar{b}$ over $M$. 
    
    In particular, it follows that Conant-independence is symmetric. Thus, by Fact \ref{fact:symmetric-conant-independence}, $T_{\mathbf{H}_4\textnormal{-free}}$ is $\NSOP_4$.
\end{proof}
\subsection{Relative Kim's lemma}
We focus on another one of the notions introduced by Mutchnik, namely, the relative Kim's lemma. In view of Remark \ref{rem:rel-kim-lemma-and-invariance}, we immediately notice that $\ind^\hyp$ cannot satisfy the relative Kim's lemma, as global $M$-$\ind^\hyp$-free types might fail to be $M$-invariant.

To fix this, let us define $\bar{a} \ind_M^\hti B$ iff $\bar{a} \ind_M^{\textnormal{i}} B$ and $\bar{a} \ind_M^\hyp B$. 
\begin{lemma} \label{lem:hti-props}
    $\ind^\hti$ is an independence relation satisfying monotonicity, full existence, and left extension.
\end{lemma}
\begin{proof}
    (Monotonicity) This holds, as it is satisfied by both $\ind^\hyp$ and $\ind^{\textnormal{i}}$. 
    
    (Full existence) Let $\bar{a}$ be a tuple, $M \models T$, and $B$ a set. Let $p := \tp(\bar{a}/M)$. By Lemma \ref{lem:ind-hti-morley-seqs-exist}, there is a global $M$-invariant extension $q$ of $p$ which is $\ind^\hyp$-free over $M$. Let $\bar{a}' \models q|_{MB}$, so that $\bar{a}' \equiv_M \bar{a}$. Then $\bar{a}' \ind^{\text{i}}_M B$ by $M$-invariance of $q$, and $\bar{a}' \ind^\hyp_M B$, since $q$ is $\ind^\hyp$-free over $M$. Hence, $\bar{a}' \ind^\hti_M B$.
    
    (Left extension) This proof is analogous to that of Lemma \ref{lem:ind-hti-morley-seqs-exist}: Suppose that $\bar{a}$, $M\models T$, and $B$ are given such that $\bar{a} \ind^\hti_M B$, and $\bar{a}'$ is a tuple extending $\bar{a}$. We may write $\bar{a}'$ as the concatenation of $\bar{a}$ with some other tuple $\bar{c}$. Let $\bar{b}$ be a (possibly infinite) tuple enumerating $B$. Fix a tuple $\bar{m}$ in $M$ of the same length as $\bar{c}$. We now want to show the consistency of the extension of $\tp(\bar{b}/M\bar{a})$ by the following set of formulas:
    \begin{enumerate}[(i)]
        \item $R(c_i, m, x_j)$ for all indices $i,j$ of the appropriate length and $m \in M$. 
        \item $\phi(\bar{c}, \bar{x}) \leftrightarrow \phi(\bar{m}, \bar{x})$ for any $\mathcal{L}$-formula $\phi$.
    \end{enumerate}
    As usual, it suffices to show that this does not generate any $\mathbf{H}_4$-structures. In this case, the proof is analogous to that of Lemma \ref{lem:ind-hti-morley-seqs-exist}, so we omit it. Pick a realisation $\bar{b}'$ of this type. By (i), we have $\bar{a}' \ind^\hyp_M \bar{b}'$, and by (ii), $\tp(\bar{a}'/M\bar{b}')$ is $M$-invariant. Therefore, letting $B'$ be the set enumerated by $\bar{b}'$, it follows that $B' \equiv_{M\bar{a}} B$ and $\bar{a}' \ind^\hti_M B'$.
\end{proof}
We can now also see that the Morley sequences appearing in Lemma \ref{lem:ind-hti-morley-seqs-exist} are precisely $\ind^\hti$-Morley sequences. So, from Proposition \ref{prop:indht-kim-ind-is-trivial}, we immediately obtain the following:
\begin{corollary}\label{cor:hti-kim-ind-is-trivial}
    For (possibly infinite) tuples $\bar{a}$ and $\bar{b}$, and a model $M \models T_{\mathbf{H}_4\textnormal{-free}}$, $\tp(\bar{a}/M\bar{b})$ does not $\ind^\hti$-Kim-divide over $M$ iff $\bar{a} \cap \bar{b} \subseteq M$.
\end{corollary}
Our goal now is to prove that $\ind^\hti$-free global types are $\leq_K$-least. At this point, $\omega$-categoricity is particularly useful:
\begin{lemma}\label{lem:omega-cat-only-inclusion-of-isolating-fmls}
    Let $T$ be $\omega$-categorical, $M \models T$, and $q$, $r$ be global $M$-invariant types with $q|_M = r|_M$. Assume that, for all formulas $\phi(\bar{x},\bar{y})$ isolating a complete type over some finite subset of $M$, if $\phi(\bar{x},\bar{y}) \in \cl_K(q)$, then $\phi(\bar{x}, \bar{y}) \in \cl_K(r)$. Then $\cl_K(q) \subseteq \cl_K(r)$.
\end{lemma}
\begin{proof}
    Let $\psi(\bar{x}, \bar{y}) \in \cl_K(q)$. Since $T$ is $\omega$-categorical, there are finitely many types over a finite set of parameters and every such type is principal, so we can write $\psi(\bar{x},\bar{y}) := \bigvee_{i<l} \phi_i(\bar{x},\bar{y})$, where each $\phi_i$ isolates a complete type over the set of parameters in $\psi$. Since $\phi_i(\bar{x},\bar{y}) \vdash \psi(\bar{x},\bar{y})$ for each $i$, it follows that $\phi_i(\bar{x},\bar{y}) \in \cl_K(q)$, and thus, by assumption, $\phi_i(\bar{x},\bar{y}) \in \cl_K(r)$ for all $i < l$. Thus, for all $M$-invariant Morley sequences $(\bar{b}_j)_{j \in \omega}$ in $r$, the set $\{\phi_i(\bar{x}, \bar{b}_j) : j \in \omega\}$ is inconsistent for all $i < l$. 

    Assume, for contradiction, that $\psi(\bar{x},\bar{y}) \notin \cl_K(r)$. So there exists some $M$-invariant Morley sequence $(\bar{c}_i)_{i \in \omega}$ in $r$ such that $\{\psi(\bar{x}, \bar{c}_i) : i \in \omega\}$ is consistent. Let $\bar{a}$ be a realisation. Then, by the pigeonhole principle, we can find an infinite subsequence $(\bar{c}_{i_j})_{j \in \omega}$ and some $k$ such that $\bar{a} \models \{\phi_k(\bar{x}, \bar{c}_{i_j}) : j \in \omega\}$. But note that $(\bar{c}_{i_j})_{j \in \omega}$ is also an $M$-invariant Morley sequence in $r$, a contradiction.
\end{proof}
This result tells us that it suffices to look at types over finite parameter sets to determine $\leq_K$-minimality.
\begin{remark} \label{rem:non-disjointness-implies-conant-div}
    Note that, for a model $M \models T_{\mathbf{H}_4\text{-free}}$ and a tuple $\bar{b}$, if $\phi(\bar{x}, \bar{b}) \vdash \bar{x} \cap \bar{b} \not\subseteq M$, then $\phi(\bar{x}, \bar{b})$ Conant-divides over $M$. In other words, for every global $M$-invariant extension $q$ of $\tp(\bar{b}/M)$, $\phi(\bar{x}, \bar{y}) \in \cl_K(q)$.
\end{remark}
\begin{lemma} \label{lem:trivial-kim-ind-gives-minimality}
    Let $q$ be a global type $\ind^\hti$-free over $M \models T_{\mathbf{H}_4\textnormal{-free}}$. For all $\bar{b} \models q|_M$ and all formulas $\phi(\bar{x}, \bar{b})$ isolating a complete type over $M'\bar{b}$ with $M' \subset M$ finite, if $\phi(\bar{x},\bar{y}) \in \cl_K(q)$, then $\phi(\bar{x}, \bar{b}) \vdash \bar{x} \cap \bar{b} \not\subseteq M$.
\end{lemma}
\begin{proof}
    Suppose that $\phi(\bar{x},\bar{b}) \vdash \bar{x} \cap \bar{b} \subseteq M$. Since $q$ is a global $M$-$\ind^\hti$-free type, by Corollary \ref{cor:hti-kim-ind-is-trivial}, $\{\phi(\bar{x},\bar{b}_i) : i \in \omega\}$ is consistent for any $(\bar{b}_i)_{i \in \omega}$ such that $\bar{b}_i \models q|_{M\bar{b}_{<i}}$ for all $i \in \omega$. This means that $\phi(\bar{x},\bar{y}) \notin \cl_K(q)$.
\end{proof}
\begin{proposition} \label{prop:indht-satisfies-kims-lemma}
    $\ind^\hti$ satisfies the relative Kim's lemma in $T_{\mathbf{H}_4\textnormal{-free}}$.
\end{proposition}
\begin{proof}
    Let $p \in S(M)$ and let $q$ be a global extension of $p$ $\ind^\hti$-free over $M$. We want to show that $q$ is a $\leq_K$-least extension of $p$. By Lemma \ref{lem:omega-cat-only-inclusion-of-isolating-fmls}, it is enough to show that, if $\phi(\bar{x},\bar{y})$ isolates a complete type over a finite set, then $\phi(\bar{x},\bar{y}) \in \cl_K(q)$ implies $\phi(\bar{x},\bar{y}) \in \cl_K(r)$ for all global $M$-invariant extensions $r$ of $p$. But notice that, by Lemma \ref{lem:trivial-kim-ind-gives-minimality}, $\phi(\bar{x},\bar{y}) \in \cl_K(q)$ implies that $\phi(\bar{x}, \bar{b}) \vdash \bar{x} \cap \bar{b} \not\subseteq M$ for any $\bar{b} \models q|_M$. Hence, by Remark \ref{rem:non-disjointness-implies-conant-div}, it follows that $\phi(\bar{x}, \bar{y}) \in \cl_K(r)$.
\end{proof}
\begin{proof}[Proof of Proposition \ref{prop:h4-free-is-tp2} (v.2)]
    By Proposition \ref{prop:h4-free-is-sop3}, $T_{\mathbf{H}_4\textnormal{-free}}$ is $\SOP_3$, and so, in particular, it is not simple. Since $\ind^\hti$ satisfies full existence and left extension by Lemma \ref{lem:hti-props} and the relative Kim's lemma by Proposition \ref{prop:indht-satisfies-kims-lemma}, it follows by Corollary \ref{cor:simple-or-tp2} that $T_{\mathbf{H}_4\textnormal{-free}}$ is $\TP_2$.
\end{proof}
This completes our second proof of Theorem \ref{thm:main}.
\begin{remark} \label{rem:no-symm-and-stationary-relation}
    Note that, by \cite[Lemma 4.3]{conant2017freeamalgamation}, there does not exist an independence relation $\ind$ over subsets of $\M \models T_{\mathbf{H}_4\text{-free}}$ satisfying invariance, full existence, symmetry, and stationarity over models, because for any pair $ab \in \M$ of distinct elements and $M \models T_{\mathbf{H}_4\textnormal{-free}}$, if $\models R(m, a, b)$ for some $m \in M$, then $\models \neg R(m, b, a)$, and so $ab \not\equiv_M ba$. This means that $T_{\mathbf{H}_4\textnormal{-free}}$ is not a free amalgamation theory, so we cannot apply Conant's methods to show $\NSOP_4$. The same argument also shows that we cannot use the criterion from \cite[Theorem 5.10]{d2025model} either. 
\end{remark}
\begin{remark}
    Mutchnik uses \cite[Theorem 3.25]{mutchnik2022conantindependence} (cf., Corollary \ref{cor:simple-or-tp2}) in the context of a theory with an independence relation $\ind$ that satisfies monotonicity, full existence, and stationarity over models. It is shown there that, if $\ind$ satisfies a generalization of freedom, known as generalized freedom, then satisfying the relative Kim's lemma and the symmetry of $\ind$-Kim-independence are equivalent, which means that we do not need to check them separately as we did here. However, it is unclear whether one can find such an independence relation in $T_{\mathbf{H}_4\text{-free}}$. 
\end{remark}
\begin{question}
    Is there an independence relation $\ind$ in $T_{\mathbf{H}_4\textnormal{-free}}$ satisfying monotonicity, full existence, and stationarity over models?
\end{question}
\bibliographystyle{amsalpha}
\bibliography{list}

\providecommand{\bysame}{\leavevmode\hbox to3em{\hrulefill}\thinspace}
\providecommand{\MR}{\relax\ifhmode\unskip\space\fi MR }
\providecommand{\MRhref}[2]{%
  \href{http://www.ams.org/mathscinet-getitem?mr=#1}{#2}
}
\providecommand{\href}[2]{#2}
\begin{thebibliography}{CHKN21}

\bibitem[AACT23]{abd2023higher}
A~Abd-Aldaim, G~Conant, and C~Terry, \emph{Higher arity stability and the functional order property}, arXiv preprint arXiv:2305.13111 (2023).

\bibitem[Adl05]{adler2005explanation}
Hans Adler, \emph{Explanation of independence}, arXiv preprint arXiv:math/0511616 (2005).

\bibitem[Ass86]{assous1986hypertournaments}
Roland Assous, \emph{Encha\^{i}nabilit\'{e} et seuil de monomorphie des tournois {$n$}-aires}, Discrete Math. \textbf{62} (1986), no.~2, 119--125. \MR{863036}

\bibitem[Bal88]{baldwin1988fundamentals}
John~T. Baldwin, \emph{Fundamentals of stability theory}, Perspectives in Mathematical Logic, Springer-Verlag, Berlin, 1988. \MR{918762}

\bibitem[BPT11]{bodirsky2011decidability}
Manuel Bodirsky, Michael Pinsker, and Todor Tsankov, \emph{Decidability of definability}, 26th {A}nnual {IEEE} {S}ymposium on {L}ogic in {C}omputer {S}cience---{LICS} 2011, IEEE Computer Soc., Los Alamitos, CA, 2011, pp.~321--328. \MR{2858903}

\bibitem[BYC14]{benyaacov2014independence}
Ita\"{\i} Ben~Yaacov and Artem Chernikov, \emph{An independence theorem for {${\rm NTP}_2$} theories}, J. Symb. Log. \textbf{79} (2014), no.~1, 135--153. \MR{3226015}

\bibitem[Cam90]{cameron1990oligomorphic}
Peter~J. Cameron, \emph{Oligomorphic permutation groups}, London Mathematical Society Lecture Note Series, vol. 152, Cambridge University Press, Cambridge, 1990. \MR{1066691}

\bibitem[CF04]{casanovas2004wei}
Enrique Casanovas and Rafel Farr\'{e}, \emph{Weak forms of elimination of imaginaries}, MLQ Math. Log. Q. \textbf{50} (2004), no.~2, 126--140. \MR{2037732}

\bibitem[CH19]{chernikov2019mekler}
Artem Chernikov and Nadja Hempel, \emph{Mekler’s construction and generalized stability}, Israel Journal of Mathematics \textbf{230} (2019), 745--769.

\bibitem[CH21]{chernikov2021n}
\bysame, \emph{On n-dependent groups and fields ii}, Forum of Mathematics, Sigma, vol.~9, Cambridge University Press, 2021, p.~e38.

\bibitem[Che]{cherlinunpublished}
Gregory Cherlin, \emph{Some classification problems for finite homogeneity}, Unpublished notes.

\bibitem[Che22]{cherlin2022multitournaments}
\bysame, \emph{Homogeneous ordered graphs, metrically homogeneous graphs, and beyond. {V}ol. {II}. 3-multi-graphs and 2-multi-tournaments}, Lecture Notes in Logic, vol.~54, Association for Symbolic Logic, Ithaca, NY; Cambridge University Press, Cambridge, 2022. \MR{4424790}

\bibitem[CHKN21]{cherlin2021hypertournaments}
Gregory Cherlin, Jan Hubi{\v{c}}ka, Mat{\v{e}}j Kone{\v{c}}n{\'y}, and Jaroslav Ne{\v{s}}et{\v{r}}il, \emph{Ramsey expansions of 3-hypertournaments}, Extended Abstracts EuroComb 2021 (Cham) (Jaroslav Ne{\v{s}}et{\v{r}}il, Guillem Perarnau, Juanjo Ru{\'e}, and Oriol Serra, eds.), Springer International Publishing, 2021, pp.~696--701.

\bibitem[CK12]{chernikov2012forking}
Artem Chernikov and Itay Kaplan, \emph{Forking and dividing in {${\rm NTP}_2$} theories}, J. Symbolic Logic \textbf{77} (2012), no.~1, 1--20. \MR{2951626}

\bibitem[Con12]{conant2012dividing}
Gabriel Conant, \emph{Dividing lines in unstable theories}, Expository article (2012).

\bibitem[Con17]{conant2017freeamalgamation}
\bysame, \emph{An axiomatic approach to free amalgamation}, J. Symb. Log. \textbf{82} (2017), no.~2, 648--671. \MR{3663421}

\bibitem[CPT19]{chernikov2019n}
Artem Chernikov, Daniel Palacin, and Kota Takeuchi, \emph{On {$n$}-dependence}, Notre Dame J. Form. Log. \textbf{60} (2019), no.~2, 195--214. \MR{3952231}

\bibitem[d'E23]{delbee2023axiomatic}
Christian d'Elb{\'e}e, \emph{Axiomatic theory of independence relations in model theory}, arXiv preprint arXiv:2308.07064 (2023).

\bibitem[DK22]{dobrowolski2022kim}
Jan Dobrowolski and Mark Kamsma, \emph{Kim-independence in positive logic}, Model Theory \textbf{1} (2022), no.~1, 55--113.

\bibitem[dMRS25]{d2025model}
Christian d'Elb{\'e}e, Isabel M{\"u}ller, Nicholas Ramsey, and Daoud Siniora, \emph{Model-theoretic properties of nilpotent groups and {L}ie algebras}, Journal of Algebra \textbf{662} (2025), 640--701.

\bibitem[EW09]{evans2009generic}
David~M. Evans and Mark Wing~Ho Wong, \emph{Some remarks on generic structures}, J. Symbolic Logic \textbf{74} (2009), no.~4, 1143--1154. \MR{2583813}

\bibitem[Hem16]{hempel2016n}
Nadja Hempel, \emph{On n-dependent groups and fields}, Mathematical Logic Quarterly \textbf{62} (2016), no.~3, 215--224.

\bibitem[Hen72]{henson1972digraphs}
C.~Ward Henson, \emph{Countable homogeneous relational structures and $\aleph_0$-categorical theories}, The Journal of Symbolic Logic \textbf{37} (1972), no.~3, 494–500.

\bibitem[JY25]{johnson2025curve}
Will Johnson and Jinhe Ye, \emph{Curve-excluding fields}, J. Eur. Math. Soc. (2025), published online first.

\bibitem[KR20]{kaplan2020kimindependence}
Itay Kaplan and Nicholas Ramsey, \emph{On {K}im-independence}, J. Eur. Math. Soc. (JEMS) \textbf{22} (2020), no.~5, 1423--1474. \MR{4081726}

\bibitem[KR23]{kruckman2023new}
Alex Kruckman and Nicholas Ramsey, \emph{A {N}ew {K}im's {L}emma}, arXiv preprint arXiv:2309.02718 (2023).

\bibitem[KRS19]{shelah2019kimindependence}
Itay Kaplan, Nicholas Ramsey, and Saharon Shelah, \emph{Local character of {K}im-independence}, Proc. Amer. Math. Soc. \textbf{147} (2019), no.~4, 1719--1732. \MR{3910436}

\bibitem[Kru19]{kruckman2019disjoint}
Alex Kruckman, \emph{Disjoint n-amalgamation and pseudofinite countably categorical theories}, Notre Dame Journal of Formal Logic \textbf{60} (2019), no.~1.

\bibitem[KS19]{kaplan2019automorphism}
Itay Kaplan and Pierre Simon, \emph{Automorphism groups of finite topological rank}, Trans. Amer. Math. Soc. \textbf{372} (2019), no.~3, 2011--2043. \MR{3976583}

\bibitem[Lac84]{lachlan1984tournaments}
A.~H. Lachlan, \emph{Countable homogeneous tournaments}, Trans. Amer. Math. Soc. \textbf{284} (1984), no.~2, 431--461. \MR{743728}

\bibitem[Li18]{li2018simplicity}
Yibei Li, \emph{Simplicity of the automorphism groups of some binary homogeneous structures determined by triangle constraints}, arXiv preprint arXiv:1806.01671 (2018).

\bibitem[Li19]{li2019automorphism}
\bysame, \emph{Automorphism groups of homogeneous structures with stationary weak independence relations}, arXiv preprint arXiv:1911.08540 (2019).

\bibitem[LP79]{lascar1979forking}
Daniel Lascar and Bruno Poizat, \emph{An introduction to forking}, J. Symbolic Logic \textbf{44} (1979), no.~3, 330--350. \MR{540665}

\bibitem[Mut22a]{mutchnik2022conantindependence}
Scott Mutchnik, \emph{Conant-independence and generalized free amalgamation}, arXiv preprint arXiv:2210.07527 (2022).

\bibitem[Mut22b]{mutchnik2023nsop2}
\bysame, \emph{On $\mathrm{NSOP}_2$ {T}heories}, arXiv preprint arXiv:2206.08512 (2022).

\bibitem[Pil83]{pillay1983stability}
Anand Pillay, \emph{An introduction to stability theory}, Oxford Logic Guides, vol.~8, The Clarendon Press, Oxford University Press, New York, 1983. \MR{719195}

\bibitem[Poi00]{poizat2000modeltheory}
Bruno Poizat, \emph{A course in model theory}, Universitext, Springer-Verlag, New York, 2000, An introduction to contemporary mathematical logic, Translated from the French by Moses Klein and revised by the author. \MR{1757487}

\bibitem[She78]{shelah1978classification}
Saharon Shelah, \emph{Classification theory and the number of nonisomorphic models}, Studies in Logic and the Foundations of Mathematics, vol.~92, North-Holland Publishing Co., Amsterdam-New York, 1978. \MR{513226}

\bibitem[She96]{shelah1996unstable}
\bysame, \emph{Toward classifying unstable theories}, Ann. Pure Appl. Logic \textbf{80} (1996), no.~3, 229--255. \MR{1402297}

\bibitem[She14]{shelah2014ndependence}
\bysame, \emph{Strongly dependent theories}, Israel J. Math. \textbf{204} (2014), no.~1, 1--83. \MR{3273451}

\bibitem[Tar23]{tartarotti2023lascar}
Mira Tartarotti, \emph{The {L}ascar {G}roup and {R}econstruction {P}roblems in {M}odel {T}heory}, Master's Thesis, University of Oxford (2023).

\bibitem[TW21]{terry2023higherorder}
C.~Terry and J.~Wolf, \emph{Higher-order generalizations of stability and arithmetic regularity}, arXiv preprint arXiv:2111.01739 (2021).

\end{thebibliography}
\end{document}